\documentclass[11pt]{article}
\usepackage{amssymb,amsmath} 
\usepackage{amsfonts,amsthm,mathrsfs}
\usepackage{cases} 
\usepackage{graphicx}
\usepackage{txfonts} 

\usepackage{xcolor}
\usepackage{color}

\definecolor{dGREEN}{rgb}{0.0,0.5,0.5}

\numberwithin{equation}{section}
\newtheorem{thm}{Theorem}[section]

\newtheorem{lemma}[thm]{Lemma}
\newtheorem{prop}[thm]{Proposition}

\newtheorem{definition}[thm]{Definition}

\begin{document}
\title{On a dryout point for a stationary incompressible thermal fluid\\ with phase transition in a pipe}
\author{Yoshikazu Giga \\ Graduate School of Mathematical Sciences \\ The University of Tokyo \\ labgiga@ms.u-tokyo.ac.jp \and 
Zhongyang Gu \\ Graduate School of Mathematical Sciences \\ The University of Tokyo \\ zgu@ms.u-tokyo.ac.jp}
\date{}
\maketitle

\begin{abstract}
A dryout point is recognized as the position where the phase transition from liquid to vapor occurs.
In the one-dimensional case, by solving the stationary incompressible Navier-Stokes-Fourier equations with phase transition, we derive a necessary and sufficient condition for a dryout point to exist when the temperature at the liquid-vapor interface is given.
In addition, we show by considering thermodynamics that the temperature at the dryout point and the density of the vapor phase can be determined by given density and sufficiently small injected mass flux of the liquid phase.
\end{abstract}

\begin{center}
\end{center}

\thispagestyle{empty}





\section{Introduction} \label{S1} 
We consider a pipe and a liquid (thermal) fluid is injected from one side of the pipe.
 We heat the pipe so that the liquid becomes a vapor.
 This type of phase transition is often appeared in industries, for example, in air conditioners.
 A steady behavior is well studied in engineering as forced convention boiling.
 The two-phase flow patterns depend on injected fluid speed, temperature and external heat from the wall; see e.g.\ G.\ F.\ Naterer \cite[S.\ 4.5.5]{Na}.
 The pattern may be affected by gravity so vertical flow and horizontal flow are discussed separately.
 However, its mathematical understanding from
 fundamental macroscopic equations like the Navier-Stokes-Fourier system for two-phase flows with phase transition proposed by \cite{PS} and earlier by \cite{IH} is still fundamentally lacking.

Among many interesting problems, in this paper we are interested in a place of the pipe where no liquid phase attached to the wall of pipe remains.
 The first place such phenomenon occurs is often called a dryout point and its location is important to control phenomena.
 The purpose of this paper is to study a dryout point from fundamental macroscopic equations; see e.g.\ \cite{PS} or \cite{PSh}.
 We consider the one-dimensional Navier-Stokes-Fourier equations for completely incompressible two-phase fluid with phase transition when the densities of phases are quite different from each other.
 We consider its stationary version and define a dryout point as a phase boundary.
 We give a simple necessary and sufficient condition for the existence of a dryout point by using latent heat, external heat, the mass flux of injected fluid and its temperature as well as thermal conductivity and specific heat at constant volume at least when the mass flux of injected fluid is sufficiently small.
By thermodynamical relation, we show that no phase transition occurs when the mass flux of injected fluid is large at least for the van der Waals' energy. 
We also give a way to calculate the location of the dryout point.
Since the problem is reduced to an ordinary differential equation for temperature, analysis itself is very easy.
However, as far as the authors know, this seems to be the first study in this direction.

Let us explain our specific setting.
 We consider a pipe
\[
	\Omega = (0,L) \times D = \left\{ (x_1,x') \bigm| 
	0 < x_1 < L,\ x' \in D \right\},
\]
where $D$ is a bounded domain in $\mathbf{R}^{n-1}$.
 A thermal fluid is injected from the entrance $\Omega_{\mathrm{in}}=\{0\}\times\Omega$.
 The wall $S=(0,L)\times\partial D$ is heated with heat flux $Q$.
 There are various types of ``steady flow'' \cite[S.\ 4.5.5]{Na}.
 For example, Figure \ref{An} is a schematic picture of annular flow.
\begin{figure}[htb]
\centering 
\includegraphics[keepaspectratio, scale=0.25]{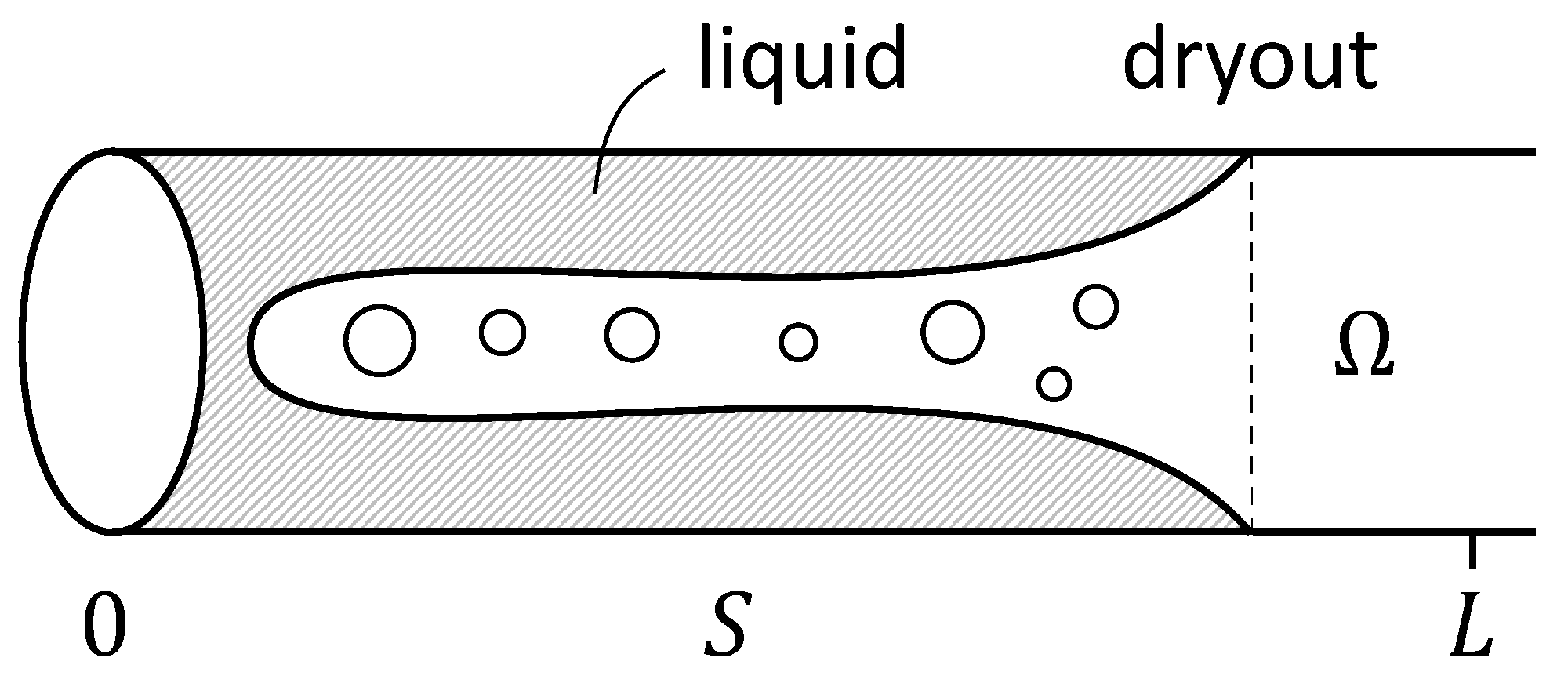}
\caption{an annular flow with a dryout point\label{An}}
\end{figure}
 The liquid region consists of a rather thin layer near the wall.
 If the external heat $Q$ is rather large, it is expected that there is no liquid region in $(L_0,L)\times D$ and the infimum of such $L_0$ is called a dryout point.

We consider the heat equation
\[
	\rho_1 \kappa_1 \left(\partial_t \theta_1 + u_1 \cdot \nabla\theta_1 \right) - d_1 \Delta\theta_1 = 0
\]
for the temperature $\theta_1$ in a liquid region.
 Here, $\rho_1$, $\kappa_1$, $d_1$, $u_1$ denote the density, the specific heat at constant volume, the heat conductivity and the velocity of the liquid phase, respectively.
We assume that $\rho_1$, $\kappa_1$, $d_1$ are positive constants.
We use the standard notation $\partial_t=\partial/\partial t$, $\nabla=(\partial_{x_1},\ldots,\partial_{x_n})$ and $\partial_{x_i}=\partial/\partial_{x_i}$.
 On the wall $S$ we impose
\[
	d_1 \frac{\partial\theta_1}{\partial\nu_{\partial\Omega}} = Q,
\]
where $\nu_{\partial\Omega}$ denotes the exterior normal of $\partial\Omega$.
 We ignore detailed descriptions of the annular flow and postulate that it occupies $(0,L_0)\times D$ since the pipe is sufficiently thin compared with the speed.
 We further assume that $u_1$ is a constant vector in the direction of $x_1$.
 We still write this constant vector $(u_1,0,0)$ by abusing the notation.
 If we take  average over $D$, we end up with
\begin{equation} \label{EAv}
	\rho_1 \kappa_1 \left(\partial_t \overline{\theta_1} + u_1 \partial_{x_1} \overline{\theta_1} \right)
	- d_1 \partial_{x_1}^2 \overline{\theta_1}
	- \frac{1}{|D|} \int_{\partial D} Q\; d\mathcal{H}^{n-2} = 0
\end{equation}
since 
\[
	\int_D \Delta' \theta_1\; dx'
	= \int_{\partial D} \frac{\partial\theta_1}{\partial\nu_{\partial D}}\; d\mathcal{H}^{n-2}
	= \frac{1}{d_1} \int_{\partial D} Q\; d\mathcal{H}^{n-2},
\]
where $\mathcal{H}^m$ denotes $m$-dimensional Hausdorff measure.
 Here $\overline{\theta_1}$ denotes the average of $\theta_1$ on $D$, i.e.,
\[
	\overline{\theta_1}(x_1) = \frac{1}{|D|} \int_D \theta_1 (x_1,x')\; dx',
\]
where $|D|$ denotes the area of $D$, i.e., $|D|=\mathcal{H}^{n-1}(D)$.
 The equation \eqref{EAv} becomes an ordinary differential equation if we consider its stationary problem.
 Its explicit form is
\[
	\rho_1 \kappa_1 u_1 \partial_{x_1} \overline{\theta_1}
	- d_1 \partial_{x_1}^2 \overline{\theta_1} = \rho_1 r_1
\]
with $r_1=\int_{\partial D}Q\;d\mathcal{H}^{n-2}\Bigm/\left(|D|\rho_1\right)$.
 Note that the heat from the wall is considered as internal heat source in this setting.
 Although this averaging process is too crude because there exists vapor region even for $x_1<L$, it is very instructive to consider one-dimensional setting to formulate a dryout point.

In this paper, we first consider the one-dimensional Navier-Stokes-Fourier equations for completely incompressible two-phase fluid with phase transition in a pipe $(0,L)$ when the liquid phase is injected from $x=0$.
 We consider internal heat source.
 If the internal heat source is large, it is expected that a vapor phase will appear.
 The question is whether or not there is a point $x_*\in(0,L)$ such that the liquid phase occupies in $\Omega_1=(0,x_*)$ and the vapor phase occupies in $(x_*,L)$ when one considers its stationary problem.
 The point $x_*$ is often called a \emph{dryout point}.
 Since the density and the velocity is constant on each phase, the problem is reduced to the two-phase stationary Stefan problem.
 However, one should note that the temperature at the interface $x_*$ is determined by thermodynamical relation at the interface.

Our goal in this paper is to give conditions for the existence of a dryout point as well as to derive its location if it exists.
 We give the density $\rho_1$ of the liquid phase and the (injected) mass flux $j_\Gamma=\rho_1u_1$, where $u_1$ is the velocity of the liquid phase which is a positive constant.
 The temperature $\theta_\mathrm{in}$ at the entrance is given which is assumed to be lower than the boiling temperature $\theta_b$ at the density $\rho_1$.
We first determine temperature $\theta_*$ at the interface $x=x_*$ and show that the density of the vapor phase is uniquely determined by $j_\Gamma$ provided that $j_\Gamma$ is sufficiently small.
 Moreover, we show that the interface temperature $\theta_*$ increases as $j_\Gamma$ increases.
 This is carried out a simple application of an implicit function theorem but the relations of pressures of both phases with $j_\Gamma\neq0$ seem to be not well-studied in the literature.
 In fact, we rewrite the Gibbs-Thomson relation and the momentum balance on the interface with no viscous stress and no surface tension effect and derive two equations for pressures of both phases.
 We also note that there is a chance that the stationary phase transition cannot occur when $j_\Gamma$ is too large.
 For this purpose, we introduce a modified mass specific Helmholtz energy
\[
	\psi^j(\rho,\theta) = \psi(\rho,\theta) - \frac{1}{2\rho^2} j_\Gamma^2
\]
and its corresponding pressure
\[
	p^j = p + \frac1\rho j_\Gamma^2,
\]
where $\psi$ denotes the mass specific Helmholtz energy and $p$ denotes the pressure defined by $p=\rho^2\partial_\rho\psi$ with density $\rho$.
Using these quantities, the momentum balance can be written as $\llbracket p^j\rrbracket=0$, while the Gibbs-Thomson relation can be written as $\left\llbracket \partial_\rho(\rho \psi^j)\right\rrbracket=0$, where $\llbracket\cdot\rrbracket$ denotes the jump between one-phase to the other.
This interpretation is a key to prove that there are no two different phases satisfying $\llbracket p^j\rrbracket=0$, $\left\llbracket \partial_\rho(\rho \psi^j)\right\rrbracket=0$ if $j_\Gamma^2$ is very large at least for the Van der Waal's energy.

We next assume that the interface temperature $\theta_*$ is given and $\theta_\mathrm{in}<\theta_*$.
We consider a semi-infinite pipe, i.e., $L=\infty$ and assume the derivative of the temperature of vapor
is bounded in $x_1$ as $x_1\to\infty$.
Then,
\begin{equation} \label{EDry} 
	(-\ell) \leq \frac{d_2 r}{\kappa_2 j_\Gamma^2}
\end{equation}
if and only if the dryout point $x_*$ exists.
 Here $-\ell$ denotes the latent heat (which is usually assumed to be positive) and $\kappa_2$, $d_2$ denote the specific heat at constant volume and the thermal conductivity of vapor phase, respectively.
 The quantity $r$ denotes the volume specific internal heat source, which is assumed to be a constant (independent of $x_1$).
 In the averaging problem, $r$ corresponds to the quantity
\[
	\int_{\partial D} Q\; d\mathcal{H}^{n-2} \biggm/ |D|.
\]
Mathematically, the problem is reduced to the one-dimensional stationary Stefan problem if the interface temperature is given.
This is a second-order ordinary differential equation with free boundary and it is easy to solve relatively explicitly.
The condition seems to be quite natural.
For example, if the mass flux is strong, then a dryout point cannot exist and liquid should stay as a liquid.
Note that the condition \eqref{EDry} is independent of the temperature at the entrance.
It is also independent of the thermal conductivity and the specific heat at constant volume of liquid phase in our model.

There is a large literature on the well-posedness if the initial boundary value problem for the Stefan problem; see e.g.\ \cite{V}, \cite{T}.
However, if we include the convection term, there seem to exist less number of articles.
The reader is referred to a recent article \cite{BCD} on this topic as well as more classical article \cite{BP} (using a renormalizing solution).
In \cite{BCD}, the convective term is also coupled with the Navier-Stokes equations with equal densities and viscosities at two-phases.
The stationary problem seems to be less studied except \cite{BL}, where the setting is quite different from ours.
The first author thanks Professor Danielle Hilhorst for letting him know the paper \cite{BCD}.

This paper is organized as follows.
 In section \ref{S2}, we recall the two-phase Navier-Stokes-Fourier system with phase transition.
 We derive one-dimensional stationary problem.
 In section \ref{S3}, we recall thermodynamics and derive temperature at liquid-vapor interface.
 In section \ref{S4}, we derive the necessary and sufficient condition such that the dryout point exists by a simple analysis of ODEs.

\section{Sharp interface model}\label{S2} 

We consider a domain $\Omega$ in $\mathbf{R}^n$, where a fluid occupies.
 We postulate that the fluid has two phases and it is bounded by an interface $\Gamma(t)$.
 More precisely, let $\Omega_i(t)$ ($i=1,2$) be an open subset of $\Omega$ bounded by a smooth hypersurface $\Gamma(t)$ such that $\Omega_1(t)\cup\Omega_2(t)\cup\Gamma(t)=\Omega$.
 Let $\nu_\Gamma$ be the unit normal vector field of $\Gamma(t)$ pointing from $\Omega_1(t)$ to $\Omega_2(t)$; see Figure \ref{Fp}.
\begin{figure}[htb]
\centering 
\includegraphics[keepaspectratio, scale=0.22]{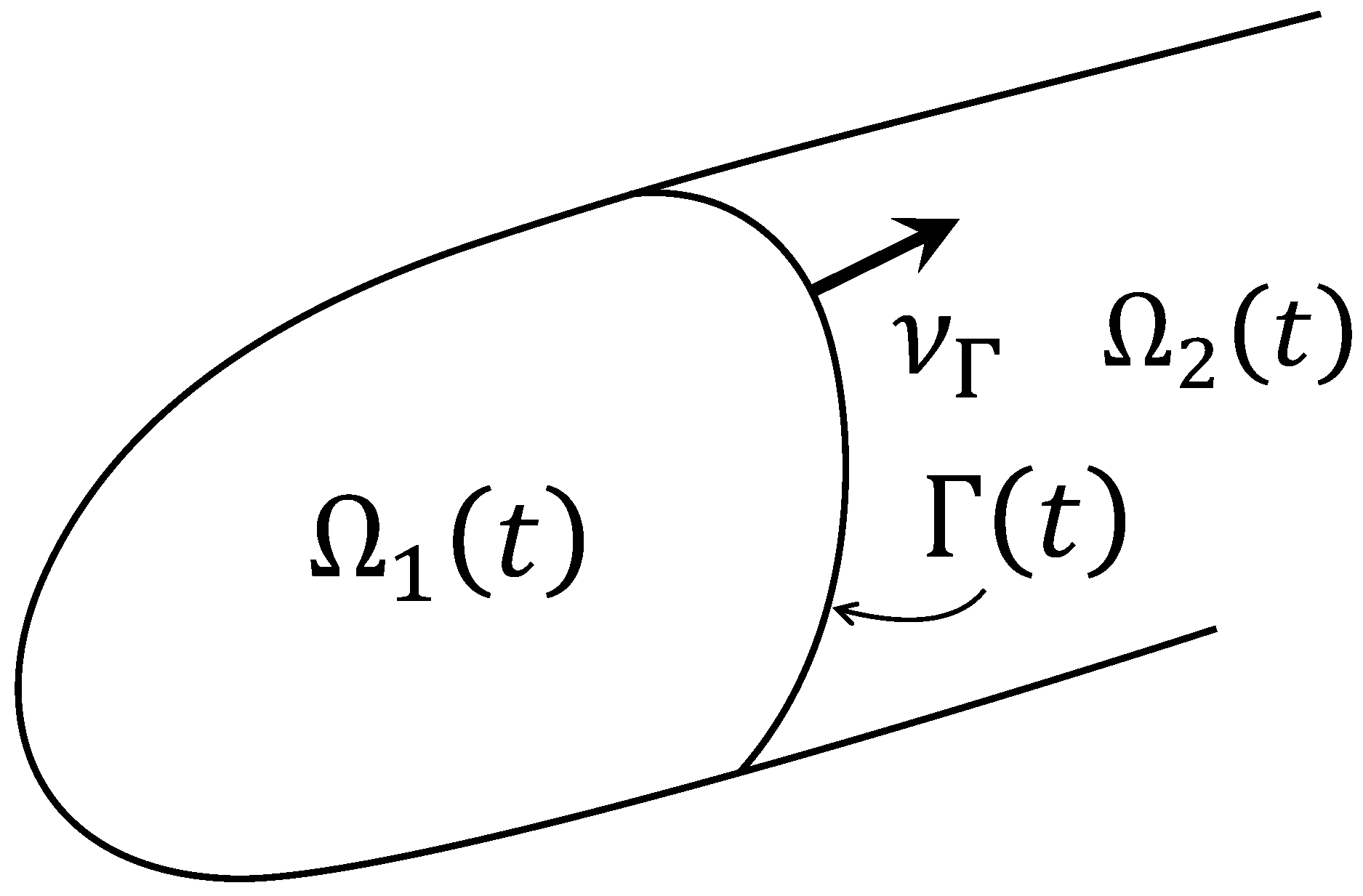}
\caption{two phases and interface in a pipe\label{Fp}}
\end{figure}
We postulate that one phase, say liquid phase of the fluid occupies in $\Omega_1(t)$ and the other phase, say vapor phase, occupies in $\Omega_2(t)$.
For a quantity $\phi$ defined in $\Omega\backslash\Gamma$, we define its jump $\llbracket\phi\rrbracket$ along $\Gamma$ by
\[
	\llbracket\phi\rrbracket(x) = \phi_2(x) - \phi_1(x), \quad
	\phi_1(x) = \lim_{\delta\downarrow0}\phi(x-\delta\nu_\Gamma), \quad
	\phi_2(x) = \lim_{\delta\downarrow0}\phi(x+\delta\nu_\Gamma) \quad \text{for} \quad x \in \Gamma(t).
\]

We consider a density field $\rho=\rho(x,t)$, a velocity field $u=u(x,t)$ and temperature field $\theta=\theta(x,t)$ to describe thermal fluid occupies in $\Omega$.
 Its restriction on each phase $\Omega_i$ is denoted by $\rho_i$, $u_i$, $\theta_i$ instead of $\rho$, $u$, $\theta$.
 We first recall the thermodynamical quantities of each phase.
 Let $\psi=\psi(\rho,\theta)$ denote the mass specific Helmholtz energy.
 It may be a different function on each phase.
 In $\Omega_i$, the mass specific Helmholtz energy is denoted by $\psi_i$ for $i=1,2$.
 We set
\begin{equation} \label{EG1}
	\eta_i = - \frac{\partial\psi_i}{\partial\theta}, \quad
	p_i = \rho^2 \frac{\partial\psi_i}{\partial\rho}.
\end{equation}
The quantity $\eta_i$ represents the mass specific entropy while $p_i$ represents the (thermodynamical) pressure.
 The mass specific internal energy $\epsilon_i$ is defined as
\[
	\epsilon_i = \psi_i + \theta\eta_i.
\]
A direct calculation shows that the relation \eqref{EG1} is equivalent to the Gibbs equation
\[
	\theta D\eta_i(\rho,\theta) = D\epsilon_i(\rho,\theta)
	+ p_i(\rho,\theta) D(1/\rho),
\]
where $D$ denotes the gradient with respect to $(\rho,\theta)$ variable.
 The quantity
\[
	\kappa_i = \frac{\partial\epsilon_i}{\partial\theta} 
\]
is called the specific heat at constant volume.

Following \cite[1.1]{PS} or \cite{PSh}, we recall conservation laws in $\Omega$.
 The mass conservation is of the form
\begin{equation} \label{EM}
	\partial_t \rho + \operatorname{div}(\rho u) = 0
	\quad\text{in}\quad \Omega\backslash\Gamma(t),
\end{equation}
where $\partial_t=\partial/\partial t$.
 The momentum conservation law with no external force is of the form 
\begin{equation} \label{EMo}
	\partial_t(\rho u) + \operatorname{div}(\rho u\otimes u) - \operatorname{div} T = 0
	\quad\text{in}\quad \Omega\backslash\Gamma(t).
\end{equation}
The energy conservation law for the energy $E=\frac12\rho|u|^2+\rho\epsilon$ is of the form
\begin{equation} \label{EE}
	\partial_t E + \operatorname{div}(Eu) - \operatorname{div} (Tu-q) = \rho r
	\quad\text{in}\quad \Omega\backslash\Gamma(t).
\end{equation}
Here $T$ denotes the symmetric stress tensor and $q$ denotes the heat flux.
 We consider a heat source $r=r_i$ on each $\Omega_i$.
 We postulate that the fluid is complete incompressible in both phases.
 Namely, $\rho_i$ is independent of $x$ and $t$.
 Then \eqref{EM} and \eqref{EMo} becomes 
\begin{equation} \label{ENS}
	\operatorname{div} u = 0, \quad
	\rho \left(\partial_t u+(u\cdot\nabla)u \right) -\operatorname{div} T = 0
	\quad\text{in}\quad \Omega\backslash\Gamma(t).
\end{equation}
We next postulate that
\[
	T = S - \pi I,
\]
where $S$ denotes the viscous stress and $\pi$ denotes a dynamic pressure.
 We also postulate that this $\pi=\pi_i$ on each phase $\Omega_i$ must equal the thermodynamic pressure on $\Gamma$.
 In other words, 
\begin{align} 
	&\pi_1 = p_1, \quad \pi_2 = p_2 
	\quad\text{on}\quad \Gamma(t),\quad\text{i.e.,} \label{EPr} \\
	&\lim_{\delta\downarrow0} \left(\pi_1(x-\nu_\Gamma\delta,t) - p_1(x-\nu_\Gamma\delta,t)\right) = 0,\quad
	\lim_{\delta\downarrow0} \left(\pi_2(x+\nu_\Gamma\delta,t) - p_2(x+\nu_\Gamma\delta,t)\right) = 0. \notag
\end{align}
For the heat flux, we postulate Fourier's law, namely,
\begin{equation} \label{EF}
	q_i = -d_i\nabla\theta_i
\end{equation}
with $d_i\geq0$ (so that it is thermodynamically consistent).
 Admitting
\eqref{EM} and \eqref{EMo}, the equation \eqref{EE} becomes
\[
	\rho_i \left(\partial_t\epsilon_i + (u\cdot\nabla) \epsilon_i\right)
	+ \operatorname{div} q_i - T_i:\nabla u_i = \rho_i r_i
	\quad\text{in}\quad \Omega_i
\]
for $i=1,2$, where $A:B=\operatorname{trace}(A^TB)$ for matrices $A$ and $B$.
 If the fluid is completely incompressible so that $\operatorname{div}u=0$, this equation becomes a kind of heat equation
\begin{equation} \label{EH1}
	\kappa_i \rho_i (\partial_t \theta_i + u_i\cdot\nabla\theta_i) -\operatorname{div} (d_i\nabla\theta_i) - S_i: \nabla u_i = \rho_i r_i
	\quad\text{in}\quad \Omega_i
\end{equation}
if we admit \eqref{EF}.
 We consider this problem in one-dimensional setting so that $u_i$ is spatially constant by $\operatorname{div}u_i=0$.
 Thus the equation becomes the heat equation with convective term, i.e.,
\begin{equation} \label{EH2}
	\kappa_i \rho_i (\partial_t \theta_i + u_i \partial_{x_1} \theta_i) - \partial_{x_1} (d_i \partial_{x_1} \theta_i) = \rho_i r_i
	\quad\text{in}\quad \Omega_i \quad (i=1,2),
\end{equation}
where $\Omega_1(t)=\left(0,x(t)\right)$, $\Omega_2(t)=\left(x(t),L\right)$, $\Gamma(t)=\left\{x(t)\right\}$.
 Here is a set of assumptions.
\begin{enumerate}
\item[(A1)] $\rho_i$ is a positive constant (independent of $x$ and $t$) for $i=1,2$ so that the fluid is completely incompressible;
\item[(A2)] $u_i$ is a scalar constant (independent of $x$ and $t$) for $i=1,2$.
\end{enumerate}
If we assume that the viscous stress $S_i=0$ for $i=1,2$, where velocity is spatially constant, $(\rho_i,u_i,\pi_i)$ solves \eqref{ENS} under (A1) and (A2) with constant $\pi_i$.

We next recall conditions on the interface $\Gamma(t)$.
 We go back to the case of general dimensions.
 It is convenient to introduce the notion of phase flux as in \cite{PS}.
 On $\Gamma(t)$, we set
\begin{align*}
	j_\Gamma^1(x,t) &= \lim_{\delta\downarrow0} \left((\rho_1 u_1)(x-\delta\nu_\Gamma,t) \cdot \nu_\Gamma(x,t)
	- \rho_1(x-\delta\nu_\Gamma,t) V_\Gamma(x,t) \right), \\
	j_\Gamma^2(x,t) &= \lim_{\delta\downarrow0} \left((\rho_2 u_2)(x+\delta\nu_\Gamma,t) \cdot \nu_\Gamma(x,t)
	- \rho_2(x+\delta\nu_\Gamma,t) V_\Gamma(x,t) \right)
\end{align*}
for $x\in\Gamma(t)$.
 Here $V_\Gamma(x,t)$ denotes the normal velocity of $\Gamma(t)$ at $x$ in the direction of $\nu_\Gamma(x,t)$.
 We postulate that there is no mass on the interface.
 Then the mass conservation law is of the form
\begin{equation} \label{ESM}
	j_\Gamma^1 = j_\Gamma^2
	\quad\text{on}\quad \Gamma(t)
\end{equation}
which will be denoted simply by $j_\Gamma$.
 In other words, there is no loss or grain of the phase flux across the interface.
 As derived in \cite[1.1]{PS}, the momentum conservation law is of the form
\begin{equation} \label{ESMo}
	\llbracket u \rrbracket j_\Gamma - \llbracket T\nu_\Gamma \rrbracket
	= \sigma H_\Gamma \nu_\Gamma
	\quad\text{on}\quad \Gamma(t).
\end{equation}
Here, we have assumed that coefficient $\sigma$ of the surface tension is a positive constant independent of $x\in\Gamma(t)$ and there is no surface viscous stress.
 The symbol $H_\Gamma$ denotes the mean curvature in the direction of $\nu_\Gamma$ of $\Gamma(t)$.
 As in \cite[1.1]{PS}, we impose constitutive relations
\begin{equation} \label{EC1}
	\llbracket u-(u\cdot\nu_\Gamma)\nu_\Gamma \rrbracket = 0
	\quad\text{on}\quad \Gamma(t).
\end{equation}
This condition says that there is no slip along $\Gamma(t)$.
 If $\llbracket\rho\rrbracket\neq0$ on any point of $\Gamma(t)$, then by using \eqref{EC1}, we can rewrite \eqref{ESM} in a concise form
\begin{equation} \label{EI1}
	\llbracket u \rrbracket = \llbracket 1/\rho \rrbracket j_\Gamma\nu_\Gamma
	\quad\text{on}\quad \Gamma(t).
\end{equation}
Under this condition, \eqref{ESMo} can be expressed as
\begin{equation} \label{EI2}
	\llbracket 1/\rho \rrbracket j_\Gamma^2 - \llbracket T\nu_\Gamma\cdot\nu_\Gamma \rrbracket = \sigma H_\Gamma, \quad
	\llbracket T\nu_\Gamma - (T\nu_\Gamma\cdot\nu_\Gamma)\nu_\Gamma \rrbracket = 0
	\quad\text{on}\quad \Gamma(t).
\end{equation}
(If $\rho_1>\rho_2$ so that $\llbracket 1/\rho \rrbracket>0$, then the first identity forces $\sigma H_\Gamma+\llbracket T\nu_\Gamma\cdot\nu_\Gamma \rrbracket\geq0$.)
 For temperature, we impose
\begin{equation} \label{EI3}
	\llbracket \theta \rrbracket = 0
	\quad\text{on}\quad \Gamma(t)
\end{equation}
which is thermodynamically consistent.
 Under this condition, the second law of the thermodynamics implies the Stefan condition
\begin{equation} \label{EI4}
	\ell j_\Gamma + \llbracket d\partial_\nu \theta \rrbracket = 0
	\quad\text{on}\quad \Gamma(t).
\end{equation}
Here $\ell=-\llbracket\theta\eta\rrbracket$ denotes the latent heat and $\partial_\nu$ is the directional derivative of $\theta$ in the direction of $\nu_\Gamma$.
 The energy conservation law implies the Gibbs-Thomson condition
\begin{equation} \label{EI5}
	\llbracket\psi\rrbracket + \left\llbracket\frac{1}{2\rho^2}\right\rrbracket  j_\Gamma^2 - \llbracket T\nu_\Gamma\cdot\nu_\Gamma/\rho\rrbracket = 0
	\quad\text{on}\quad \Gamma(t).
\end{equation}
Note that \eqref{EI5} is unnecessary if we assume that there occurs no phase transition, i.e., $j_\Gamma\equiv0$ to close the system.
 For detailed derivation of \eqref{EI4} and \eqref{EI5}, see \cite[1.1]{PS} or \cite{PSh}.
 In \cite{PSh}, general thermodynamically consistent constitutive laws are discussed instead of \eqref{EC1} and \eqref{EI3}.

In summary, we consider conditions \eqref{EI1}--\eqref{EI5} on $\Gamma(t)$ under the kinematic condition $V_\Gamma=u\cdot\nu_\Gamma-j_\Gamma/\rho$, where $V_\Gamma$ denotes the normal velocity of $\Gamma(t)$.
We shall rewrite \eqref{EI5} in a form that does not include $j_\Gamma$.

\begin{prop} \label{PGib1}
Assume the first relation of \eqref{EI2}.
 Then \eqref{EI5} is equivalent to
\begin{equation} \label{EI5M}
	\llbracket\psi\rrbracket - \left\llbracket\frac{1}{\rho}\right\rrbracket \left( \frac{T_1^\nu+T_2^\nu}{2} \right) + \frac{1}{\rho_*} \sigma H_\Gamma = 0.
\end{equation}
Here, $T_i^\nu=T_i\nu_\Gamma\cdot\nu_\Gamma$ ($i=1,2$) and $\rho_*$ is the harmonic mean of $\rho_i$ i.e., $1/\rho_*=(1/\rho_1+1/\rho_2)/2$.
\end{prop}

\begin{proof}
We first note that
\[
	\left\llbracket\frac{1}{2\rho^2}\right\rrbracket = \frac{1}{2\rho_2^2} - \frac{1}{2\rho_1^2}
	= \frac12 \left(\frac{1}{\rho_1}+\frac{1}{\rho_2}\right) \left(\frac{1}{\rho_2}-\frac{1}{\rho_1}\right)
	= \frac{1}{\rho_*} \left\llbracket\frac{1}{\rho}\right\rrbracket.
\]
Plugging the first relation of \eqref{EI2} into \eqref{EI5}, we obtain
\begin{equation} \label{EIM}
	\llbracket\psi\rrbracket + \frac{1}{\rho_*} \left(\llbracket T^\nu\rrbracket + \sigma H_\Gamma \right) - \left\llbracket\frac{T^\nu}{\rho}\right\rrbracket = 0.
\end{equation}
By definition of $\rho_*$, we see that
\[
	\frac{1}{\rho_*} - \frac{1}{\rho_2} 
	= - \left( \frac{1}{\rho_*} - \frac{1}{\rho_1} \right) 
	= \frac12 \left( \frac{1}{\rho_1} -\frac{1}{\rho_2}\right).
\]
Thus,
\begin{align*}
	\frac{1}{\rho_*} \llbracket T^\nu\rrbracket - \left\llbracket\frac{T^\nu}{\rho}\right\rrbracket 
	&= T_2^\nu \left( \frac{1}{\rho_*} - \frac{1}{\rho_2} \right)
	- T_1^\nu \left( \frac{1}{\rho_*} - \frac{1}{\rho_1} \right) \\
	&= \frac{T_1^\nu+T_2^\nu}{2} \left( \frac{1}{\rho_1} - \frac{1}{\rho_2} \right)
	= - \left\llbracket\frac{1}{\rho}\right\rrbracket \frac{T_1^\nu+T_2^\nu}{2}.
\end{align*}
The relation \eqref{EIM} now yields \eqref{EI5M}.
 This also shows the converse implication from \eqref{EI5M} to \eqref{EI5}.  
\end{proof}

If $\Omega$ is a bounded domain and $\Gamma(t)$ does not touch the boundary $\Omega$, the equation \eqref{ENS} with $S=\mu(\nabla u+\nabla u^T)$, $\mu>0$ and \eqref{EH1} under \eqref{EI1}, \eqref{EI2}, \eqref{EI3}, \eqref{EI4}, \eqref{EI5} is known to be well-posed under suitable boundary condition on $\partial\Omega$ and a suitable assumption on $\psi$ \cite{PS}.
 In other words, for a given initial velocity $u_i$, temperature $\theta_i$ ($i=1,2$) and a given interface $\Gamma(0)$, there is a local-in-time solution.
 If $\Gamma(t)$ touches $\partial\Omega$, there is few literature.
 In \cite{Wi}, \cite{Wa}, it is shown that the problem is well-posed (locally-in-time) if the contact angle is 90 degrees and $\Omega$ is a finite straight cylinder with bounded cross-section $D$, i.e., $\Omega=\left\{x'\in D,\ |x_n|<L\right\}$ and $\Gamma(t)$ is given as the graph of a function.
 No phase transition is assumed to occur.

It is sometimes convenient to introduce motion of the mass specific volume $v$ defined by $v=1/\rho$.
 The quantity $1/\rho_*$ is nothing but the average of $v_1$ and $v_2$, where $v_i=1/\rho_i$.
 The quantity $\llbracket1/\rho\rrbracket$ is the difference of the mass specific volume, i.e.,
\[
	\llbracket 1/\rho\rrbracket = \llbracket v\rrbracket = v_2-v_1.
\]

\begin{prop} \label{PGib2}
Assume further that $\sigma H_\Gamma=0$ and $T_i=-p_iI$ at $x\in\Gamma(t)$.
Then \eqref{EI5M} is of the form
\begin{equation} \label{EI5P}
	\llbracket\psi\rrbracket + \llbracket v \rrbracket \frac{p_1+p_2}{2} = 0
	\quad\text{at}\quad x \in \Gamma(t).
\end{equation}
\end{prop}

In the case $T_i=-p_iI$ on $\Gamma(t)$ and $\sigma H_\Gamma=0$ on $\Gamma(t)$, we have a relation \eqref{EI5P} between pressure $p_1$, $p_2$ and $\theta$ as well as $\rho_1$, $\rho_2$.
 From the first relation of \eqref{EI2}, we have
\begin{equation} \label{EI2S}
	j_\Gamma^2 = \frac{p_1-p_2}{v_2-v_1}.
\end{equation}
If $j_\Gamma^2$ is given, the temperature $\theta$ and one of $\rho_i$ on $\Gamma(t)$ are determined since $p_i$ is determined by the temperature $\theta$ and the other $\rho_i$ under suitable assumptions on $\psi$.
 We shall discuss this thermodynamical discussion in the next section.

We now come back to one-dimensional setting under the assumption (A1), (A2).
 We consider \eqref{EI1}--\eqref{EI5} on the interface $x=x(t)$.
 The mass conservation law \eqref{ESM} is of the form
\begin{equation} \label{EI1S}
	\rho_1 \left(u_1 - \dot{x}(t)\right) = \rho_2 \left(u_2 - \dot{x}(t)\right) = j_\Gamma.
\end{equation}
The momentum conservation law \eqref{EI2} becomes \eqref{EI2S}.
 The Gibbs-Thomson condition becomes \eqref{EI5P} and $T_i=-p_iI$.
 In the bulk, the equation \eqref{EH2} is of the form
\begin{align}
\begin{aligned} \label{EH3}
	\kappa_1\rho_1(\partial_t\theta_1 + u_1\partial_{x_1} \theta_1) - \partial_{x_1} (d_1 \partial_{x_1} \theta_1) = \rho_1 r_1
	\quad\text{in}\quad \Omega_1 = \left(0,x(t)\right), \\
	\kappa_2\rho_2(\partial_t\theta_2 + u_2\partial_{x_1} \theta_2) - \partial_{x_1} (d_2 \partial_{x_1} \theta_2) = \rho_2 r_2
	\quad\text{in}\quad \Omega_2 = \left(x(t),L\right).
\end{aligned}
\end{align}
At the interface $x=x(t)$, we have, by \eqref{EI3}, i.e.,
\begin{equation} \label{EI3S}
	\theta_1 \left(x(t),t\right) = \theta_2 \left(x(t),t\right).
\end{equation}
We further impose the Stefan condition \eqref{EI4} 
\begin{equation} \label{EI4S}
	\ell j_\Gamma + d_2 \partial_{x_1} \theta_2 - d_1 \partial_{x_1} \theta_1 = 0
	\quad\text{at}\quad x = x(t)
\end{equation}
and the Gibbs-Thomson condition \eqref{EI5P}, i.e.,
\begin{equation} \label{EI5S}
	\psi_2(\rho_2,\theta_2) - \psi_1(\rho_1,\theta_1) + (v_2-v_1) \frac{p_1+p_2}{2} = 0
	\quad\text{at}\quad x = x(t),
\end{equation}
where $v_i=1/\rho_i$ ($i=1,2$).
 The phase flux $j_\Gamma$ satisfies the mass conservation law \eqref{EI1S} and its relation with the pressure is given \eqref{EI2S} which is the momentum conservation law.
 This is a two-phase Stefan type free boundary problem.
 At the boundary of $\Omega$, we impose the Dirichlet condition at the left
\begin{equation} \label{EB1}
	\theta(0,t) = \theta_\mathrm{in}
\end{equation}
and the Neumann condition
\begin{equation} \label{EB2}
	\frac{\partial\theta}{\partial x_1} (L,t) = 0
\end{equation}
on the right if $L<\infty$.
 If $L=\infty$, we impose that $\partial\theta/\partial x_1$ is spatially bounded as $x_1\to\infty$.

Let us clarify the problem under (A1) and (A2) when $\psi_i$, $r_i$ and $d_i$ are given for $i=1,2$.
 Both densities $\rho_1$, $\rho_2$ and the velocities $u_1$, $u_2$ are constants.
 We further assume that $\rho_1$ is a given positive constant with $\rho_1\neq\rho_2$ and the entrance velocity $u_1$ is given.
 Our problem becomes \eqref{EI2S}, \eqref{EI1S}, \eqref{EH3}, \eqref{EI3S}, \eqref{EI4S}, \eqref{EI5S} with \eqref{EB1} and \eqref{EB2} supplemented with initial conditions for the temperature $\theta$ and the location of interface, namely  
\begin{align*} 
	& x(0) = x_0 \in (0,L), \quad \theta_1(x_1,0) = \theta_{10}(x_1)\ \ (0<x_1<x_0), \\
	& \theta_2(x_1,0) = \theta_{20}(x_1)\ \ (x_0<x_1<L) \quad\text{with}\quad \theta_{10}(x_0) = \theta_{20}(x_0).
\end{align*}
The unknown are functions $\theta_1$, $\theta_2$ and constants $\rho_2$ and $u_2$ together with the location of the interface.
 Note that $\psi_1$ and $\psi_2$ are given so that $p_1$ and $p_2$ are given as a function of $\rho_i$ and $\theta_i$ ($i=1,2$).
 The pressure $\pi_i$ in the both phases are constant and it is determined by $p_i$ at the interface.
 The interface temperature $\theta_*$ is determined by \eqref{EI4} and \eqref{EI2S}.
 This is one-dimensional version of \eqref{ENS}, \eqref{EPr}, \eqref{EF} with \eqref{EI1}--\eqref{EI5}.
 If the interface temperature $\theta_*$ is given, this is a classical two-phase Stefan problem with drift term.
 The well-posedness of such a type of Stefan problems is well-studied when there is no drift term; see \cite{V}, \cite{T} and references therein.
 However, it is difficult to find literature including our setting.   

We are interested in a stationary problem so that $\theta_i$ and $x=x(t)$ is time independent.
 Then $u_2$ is determined by \eqref{EI1S} as
\begin{equation} \label{EI1SS}
	\rho_1 u_1 = \rho_2 u_2 (=j_\Gamma).
\end{equation}
Let $x=x_*$ be the location of the interface and $\theta_*$ be the temperature at the interface.
 Then, \eqref{EH3} with \eqref{EI1SS}, \eqref{EI2S}, \eqref{EI3S}, \eqref{EI4S}, \eqref{EI5S} becomes
\begin{align*} 
	& \kappa_1 \rho_1 u_1 \partial_{x_1} \theta_1 - \partial_{x_1}(d_1 \partial_{x_1} \theta_1) = \rho_1 r_1
	\quad\text{in}\quad (0,x_*), \\
	& \kappa_2 \rho_2 u_2 \partial_{x_1} \theta_2 - \partial_{x_1}(d_2 \partial_{x_1} \theta_2) = \rho_2 r_2
	\quad\text{in}\quad (x_*,L), \\
	& u_2 = \rho_1 u_1/\rho_2, \\
	& (\rho_1 u_1)^2 = \left(p_1(\rho_1,\theta_*) - p_2(\rho_2,\theta_*)\right) / (v_2-v_1), \\
	& \theta_1 = \theta_2 = \theta_* \quad\text{at}\quad x = x_*, \\
	& \ell \rho_1 u_1 + d_2 \partial_{x_1} \theta_2 - d_1 \partial_{x_1} \theta_1 = 0
	\quad\text{at}\quad x=x_*, \\
	& \psi_2 (\rho_2, \theta_\ast) - \psi_1(\rho_1,\theta_*) + (v_1-v_2) \frac{p_1(\rho_1,\theta_*)+p_2(\rho_2,\theta_*)}{2} = 0
	\quad\text{at}\quad x=x_*.
\end{align*}
We impose boundary conditions \eqref{EB1}, \eqref{EB2} to this system.
 If $L=\infty$, then we impose that $\partial_{x_1}\theta_2$ is bounded as $x_1\to\infty$.
 For a given $\rho_1$, $u_1$, $r_i$ ($i=1,2$) and $\theta_1(0)=\theta_\mathrm{in}$, our problem is to find constants $x_*$, $\rho_2$, $u_2$ and functions $\theta_1$, $\theta_2$ solving the above system.
 The problem is now the stationary Stefan problem but one should note that the interface temperature $\theta_*$ is also unknown.
 Fortunately, this problem is decoupled and can be discussed separately as in the next section.
\begin{definition} \label{DDR}
The point $x_*$ is called a \emph{dryout point}.
\end{definition}

\section{Interface temperature} \label{S3} 

If $j_\Gamma$ is given and small, then the Gibbs-Thomson law \eqref{EI5P} (or \eqref{EI5S}) and the momentum balance \eqref{EI2S} determines the temperature $\theta_*$ at the interface under a suitable condition on the Helmholtz energy.
In this section, we give a few sufficient conditions so that $\theta_*$ is uniquely determined by $\psi_i$'s and one of $\rho_i$'s.

We recall conventional assumptions on the mass specific Helmholtz energy $\psi$.
It turns out that it is convenient to write $\psi$ in the mass specific volume $v=1/\rho$.
We set
\[
	\tilde{\psi}(v,\theta) := \psi(1/v,\theta).
\]
The merit of this expression is that the pressure $\tilde{p}(v,\theta):=p(1/v,\theta)$ is of the form
\[
	\tilde{p}(v,\theta) = -\partial_v \tilde{\psi}(v,\theta)
\]
since
\[
	\partial_v \tilde{\psi}(v,\theta) = \partial_\rho \psi \cdot \partial_v (1/v)
	= -(1/v^2) \partial_\rho\psi = -\rho^2 \partial_\rho\psi = -p(1/v,\theta).
\]
Here $\tilde{f}$ denotes the function of $f$ in the variable $v$.
 Let $\eta$ be the mass specific entropy.
 It is defined as $\eta=-\partial_\theta\psi$ so that
\[
	\tilde{\eta}(v,\theta) = -\partial_\theta \tilde{\psi}(v,\theta).
\]
In other words, $(\tilde{p},\tilde{\eta})$ is simply a gradient of $-\tilde{\psi}$.
 We list a few basic assumptions which are physically reasonable.
\begin{enumerate}
\item[(R)] (Regularity) $\tilde{\psi}\in C^2\left((\alpha,\infty)\times(0,\infty)\right)$ with some $\alpha>0$;
\item[(M1)] (Monotonicity of the pressure) $\partial_v\tilde{p}<0$ and $\partial_\theta\tilde{p}=\partial_v\tilde{\eta}>0$ in $(\alpha,\infty)\times(0,\infty)$.
 In particular, the pressure $p=p(\rho,\theta)$ is strictly monotone increasing in both variables.
\item[(M2)] (Monotonicity of the entropy in the temperature) $\partial_\theta\tilde{\eta}>0$ in $(\alpha,\infty)\times(0,\infty)$.
\end{enumerate}
The condition (M2) is equivalent to saying that the specific heat at constant volume
\[
	\kappa = \partial_\theta \epsilon
\]
is positive everywhere, where $\epsilon=\psi+\eta\theta$ denotes the mass specific internal energy.
 Indeed,
\[
	\kappa = \partial_\theta \epsilon
	= \partial_\theta \psi + \eta + \theta \partial_\theta \eta
	= \theta \partial_\theta \eta, \quad
	\big( \partial_\theta \eta \big) (\rho, \theta) = \big( \partial_\theta \widetilde{\eta} \big) (v, \theta).
\]
The assumption $\partial_v\tilde{p}<0$ in (M1) implies that $\tilde{\psi}$ is convex in $v$ while (M2) implies that $\tilde{\psi}$ is concave in $\theta$.

Let $\psi_\ell$ denote the Helmholtz energy of the liquid phase while $\psi_g$ denote the Helmholtz energy of the vapor (gas) phase.
 The domain of definition is $(\alpha_\ell,\infty)\times(0,\infty)$ and $(\alpha_g,\infty)\times(0,\infty)$, respectively.
 We assume
\[
	\alpha_\ell \geq \alpha_g
\]
and (R), (M1), (M2) for each $\tilde{\psi}=\tilde{\psi}_\ell$, $\tilde{\psi}_g$.

Here is a structural assumption to describe a phase transition.
\begin{enumerate}
\item[(P)] Let $\theta_c>0$ be the critical temperature so that $\tilde{\psi}_u=\min(\tilde{\psi}_\ell,\tilde{\psi}_g)=\tilde{\psi}_\ell\wedge\tilde{\psi}_g$ is convex in $v\geq\alpha_\ell(\geq\alpha_g)$ for $\theta\geq\theta_c$ and for $\theta<\theta_c$, $\tilde{\psi}_u=\tilde{\psi}_\ell$ for $v<v_\theta$ and $\tilde{\psi}_u=\tilde{\psi}_g$ for $v\geq v_\theta$ with some $v_\theta$ depending on $\theta$.
 We further assume that $v_\theta$ is strictly decreasing in $\theta$ so that $\tilde{p}_u=-\partial_v\tilde{\psi}_u$ is increasing in $\theta$; see Figure \ref{Fgu}, \ref{Fpr}.
\end{enumerate}
If $\tilde{\psi}_\ell$ and $\tilde{\psi}_g$ satisfies (R), (M1), (M2) with $\alpha_\ell\geq\alpha_g$, under (P) our unified Helmholtz energy $\tilde{\psi}_u$ satisfies (M2) and $\partial_\theta\tilde{p}_u>0$ with $C^2$ regularity outside $v=v_\theta$ for $\theta<\theta_c$.
\begin{figure}[htb]
  \begin{minipage}[b]{0.47\linewidth}
\centering 
\includegraphics[keepaspectratio, scale=0.25]{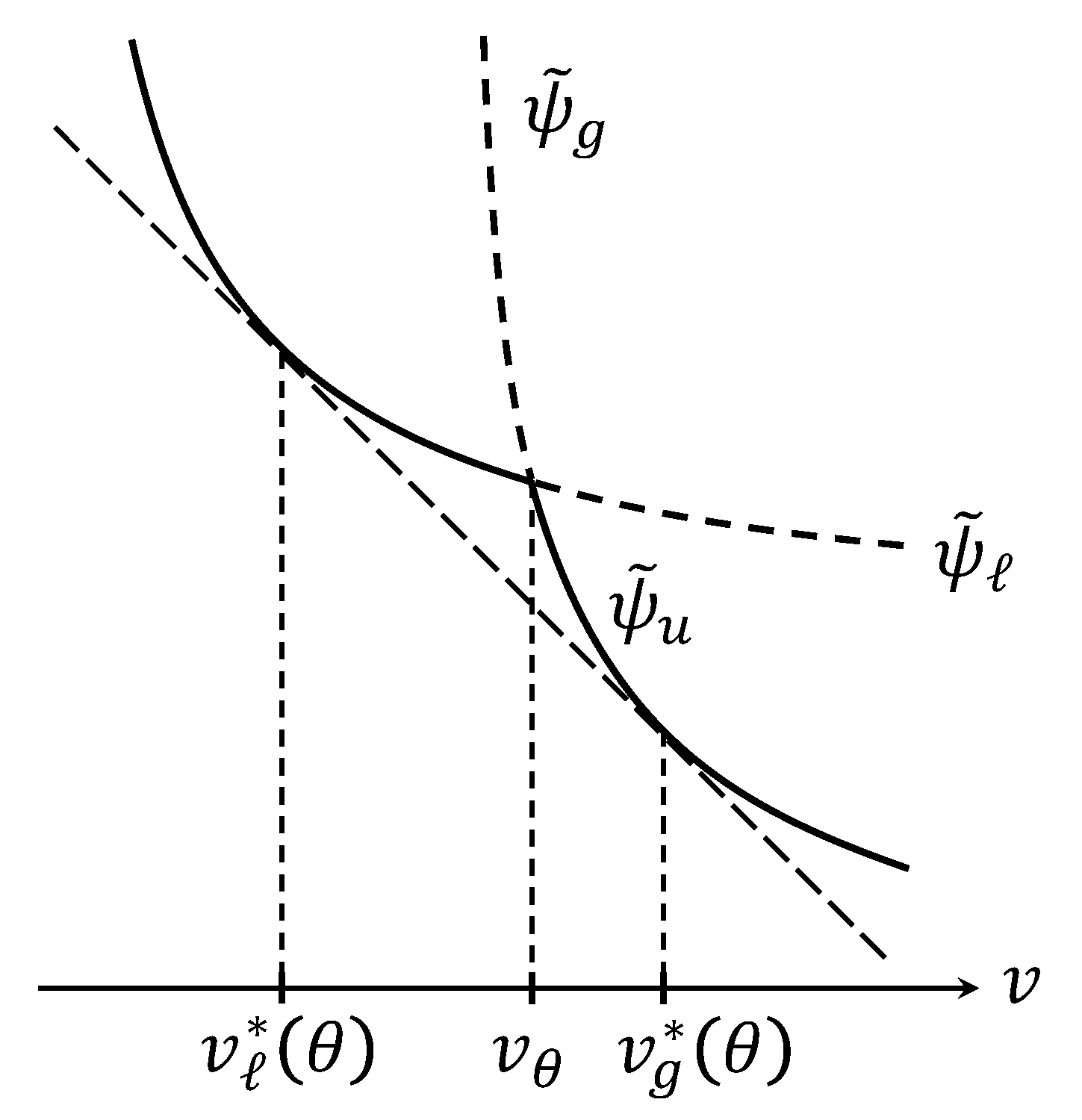}
\caption{the graph of $\tilde{\psi}_u(\cdot,\theta)$ for $\theta<\theta_c$ and bitangent line\label{Fgu}}
	\end{minipage}
	  \begin{minipage}[b]{0.03\linewidth}
	  \centering 
\includegraphics[keepaspectratio, scale=0.03]{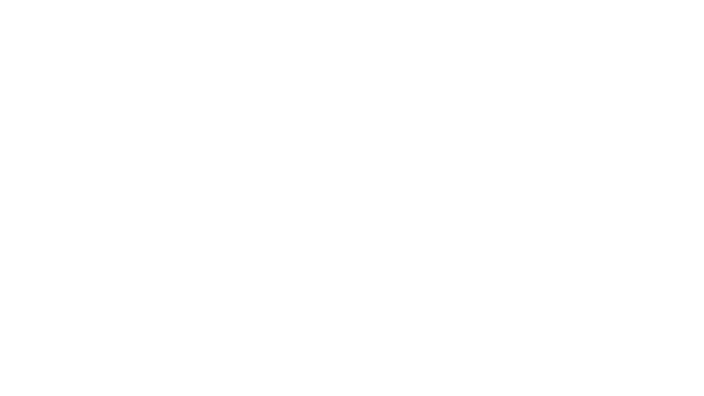}
	  \end{minipage}
  \begin{minipage}[b]{0.47\linewidth}
\centering 
\includegraphics[keepaspectratio, scale=0.25]{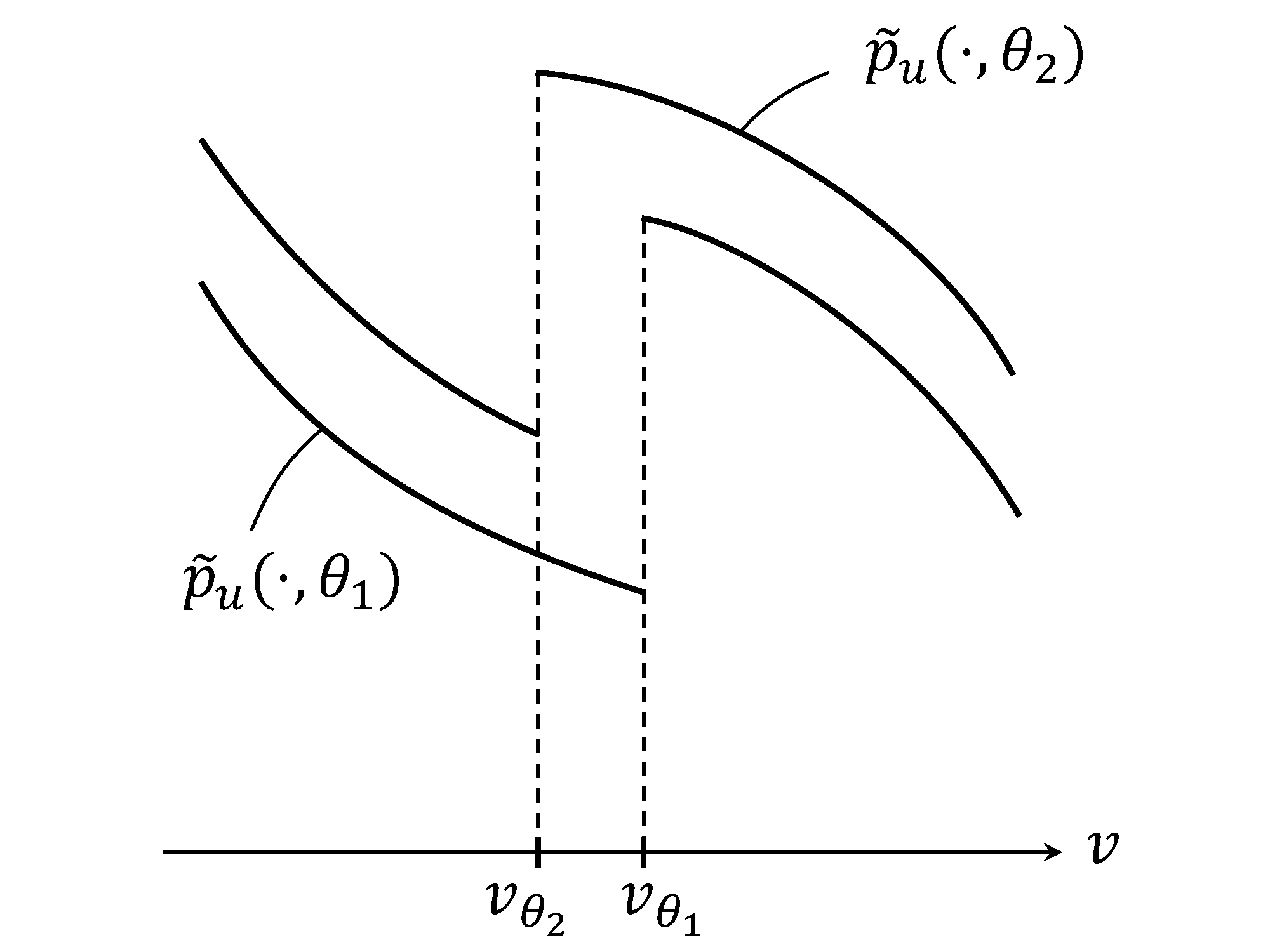}
\caption{the graph of the pressure $\tilde{p}_u=-\partial_v\tilde{\psi}_u$ under $\theta_1<\theta_2$\label{Fpr}}
  \end{minipage}
\end{figure}

We consider the convexification of $\tilde{\psi}_u(\cdot,\theta)$ in $v$ for $\theta<\theta_c$.
 There is a bitangent line in some interval $\left[v_\ell^*(\theta), v_g^*(\theta)\right]$
 to $\tilde{\psi}_u$ whose slope is $-\tilde{p}\left(v_\ell^*(\theta),\theta\right) = -\tilde{p}\left(v_g^*(\theta),\theta\right)$ and outside this interval the convexification of $\tilde{\psi}_u$ agrees with $\tilde{\psi}_u$; see Figure \ref{Fgu}.
 For a given pressure $p_*$, we call $\rho_g^*=1/v_g^*$ a \emph{saturated gas density} if $p_*=\tilde{p}\left(v_g^*(\theta),\theta\right)$ at the temperature $\theta<\theta_c$.
 Similarly, if $p_*=\tilde{p}\left(v_\ell^*(\theta),\theta\right)$, we call $\rho_\ell^*=1/v_\ell^*$ a \emph{saturated liquid density} at the temperature $\theta<\theta_c$.
 Such a $\theta$ is often called a \emph{boiling temperature}.
 For a given temperature $\theta$, the
pressure $p_*$ is called a \emph{saturated vapor pressure}, i.e., the pressure satisfying $p_*=\tilde{p}\left(v_g^*(\theta),\theta\right)=\tilde{p}\left(v_\ell^*(\theta),\theta\right)$.
 Here is a set of standard assumptions.
\begin{enumerate}
\item[(S1)] The function $p_*(\theta)$ is strictly increasing in $\theta<\theta_c$.
\item[(S2)] Moreover, $v_g^*(\theta)$ is strictly decreasing and $v_\ell^*(\theta)$ is strictly increasing in $\theta$.
\end{enumerate}
These assumptions are consistent with (P).
 The monotonicity (S1) of $p_*$ in $\theta$ is often derived from the Clasius-Clapeyron relation.
 Here is our interpretation.
 At the saturated pressure the energy balance \eqref{EI5P} holds with $p_1=p_2=p_*$.
 More precisely, if $\Omega_1$ is liquid phase and $\Omega_2$ is a gas phase, \eqref{EI5P} or \eqref{EI5S} reads
\[
	\tilde{\psi}_g \left(v_g^*(\theta),\theta\right)
	- \tilde{\psi}_\ell \left(v_\ell^*(\theta),\theta\right)
	+ \left(v_g^*(\theta) - v_\ell^*(\theta) \right) p_* = 0,
\]
where the continuity of the temperature \eqref{EI3S} is assumed. 
 Differentiating in $\theta$, we see that
\begin{gather*}
	- \tilde{p} \left(v_g^*(\theta),\theta\right) \frac{dv_g^*}{d\theta}
	+ \tilde{p} \left(v_\ell^*(\theta),\theta \right) \frac{dv_\ell^*}{d\theta}
	- \tilde{\eta} \left(v_g^*(\theta),\theta\right)
	+ \tilde{\eta} \left(v_\ell^*(\theta),\theta\right) \\
	+ \frac{dv_g^*}{d\theta} p_* - \frac{dv_\ell^*}{d\theta} p_*
	+ \frac{dp_*}{d\theta} (v_g^* - v_\ell^*) = 0.
\end{gather*}
Since $\tilde{p}\left(v_g^*(\theta),\theta\right)=\tilde{p}\left(v_\ell^*(\theta),\theta\right)=p_*$, we get
\[
	-\tilde{\eta} \left(v_g^*(\theta),\theta\right)
	+ \tilde{\eta} \left(v_\ell^*(\theta),\theta\right)
	+ \frac{dp_*}{d\theta} (v_g^* - v_\ell^*) = 0.
\]
In other words,
\begin{equation} \label{ECC}
	\frac{dp_*}{d\theta} = \frac{\eta_g^*-\eta_\ell^*}{v_g^* - v_\ell^*}
	= \frac{(-\ell)}{\theta(v_g^* - v_\ell^*)}, \quad
	\eta_g^* = \tilde{\eta}\left(v_g^*(\theta),\theta\right), \quad
	\eta_\ell^* = \tilde{\eta}\left(v_\ell^*(\theta),\theta\right),
\end{equation}
where $\ell=-\theta(\eta_g^*-\eta_\ell^*)$ denotes the latent heat in \eqref{EI4}.
 Here we invoke \eqref{EI3S}.
 The relation \eqref{ECC} is nothing but the
Clasius-Clapeyron relation.
We conclude from \eqref{ECC} 
\begin{prop} \label{PCC}
If $-\ell>0$ and $v_g^*>v_\ell^*$, then $dp_*/d\theta>0$.
In particular, (S1) follows.
\end{prop}
We are interested to determine the temperature of the interface $\Gamma$ assuming that $\Gamma$ is flat.
The interface $\Gamma$ is assumed to bound a liquid phase $\Omega_1$ with density $\rho_\ell$ and a gas phase $\Omega_2$ with density $\rho_g$.
Here $\rho_\ell$ is given positive constant.
We shall determine interface temperature $\theta$ and $\rho_g(>\rho_\ell)$ at least when $j_\Gamma$ is small.
We write the corresponding mass specific volume by $v_g=1/\rho_g$, $v_\ell=1/\rho_\ell$.
By the momentum conservation (\ref{EI2S}) across $\Gamma$, we have
\[
	p_\ell - p_g = j_\Gamma^2(v_g-v_\ell),
\]
where $p_\ell$ is the pressure of the liquid on $\Gamma$ while $p_g$ denote the pressure of the gas phase on $\Gamma$.
 We set $Z=j_\Gamma^2/2$.
 We write \eqref{EI5P} or \eqref{EI5S} of the form
\[
	\frac{p_\ell + p_g}{2} = \frac{1}{v_g - v_\ell}
	\left(-\tilde{\psi}_g(v_g,\theta) + \tilde{\psi}_\ell(v_\ell,\theta)\right).
\]
Using $p_\ell-p_g=2Z(v_g-v_\ell)$, we observe that
\begin{gather}
	\tilde{p}_\ell(v_\ell,\theta) = Z(v_g-v_\ell) + \frac{1}{v_g-v_\ell}
	\left(-\tilde{\psi}_g(v_g,\theta) + \tilde{\psi}_\ell(v_\ell,\theta)\right), \label{ESt} \\
		\tilde{p}_g(v_g,\theta) = -Z(v_g-v_\ell) + \frac{1}{v_g-v_\ell}
	\left(-\tilde{\psi}_g(v_g,\theta) + \tilde{\psi}_\ell(v_\ell,\theta)\right). \label{ESt2}
\end{gather}
We fix $v_\ell$ and write a system of equation for $v=v_g$ and $\theta$.
We set
\begin{gather*}
	f_1(\theta,v,Z) := \tilde{p}_\ell(v_\ell,\theta) - Z(v-v_\ell) - \frac{1}{v-v_\ell} \left(\tilde{\psi}_\ell(v_\ell,\theta) - \tilde{\psi}_g(v,\theta)\right), \\
	f_2(\theta,v,Z) := \tilde{p}_g(v,\theta) + Z(v-v_\ell) - \frac{1}{v-v_\ell} \left(\tilde{\psi}_\ell(v_\ell,\theta) - \tilde{\psi}_g(v,\theta)\right).
\end{gather*}
Then the equation \eqref{ESt}, \eqref{ESt2} for $(v,\theta)$ is of the form
\begin{equation} \label{ESE}
	f_1(v,\theta,Z) = 0, \quad f_2(v,\theta,Z) = 0.
\end{equation}
We assume (R), (M1), (M2) for $\psi_g$ and $\psi_\ell$.
We further assume (S1) and (S2).
For $v_\ell=1/\rho_\ell$, let $\theta_b$ be its boiling temperature, i.e., $v_\ell=v_\ell^*(\theta_b)$.
Since at the boiling temperature
\[
	\tilde{p}_\ell(v_\ell,\theta_b) = \frac{1}{v_g^*-v_\ell} \left(\tilde{\psi}_\ell(v_\ell,\theta_b) - \tilde{\psi}_g(v_g^*,\theta_b)\right)
	= \tilde{p}_g(v_g^*,\theta_b)
\]
with $v_g^*=v_g^*(\theta_b)$, we see that
\[
	f_1(\theta_b,v_g^*,0) = 0, \quad f_2(\theta_b,v_g^*,0) = 0.
\]
We shall solve \eqref{ESE} near $Z=0$ by the implicit function theorem.
We calculate
\begin{align*} 
	&\partial_\theta f_1 = \partial_\theta \tilde{p}_\ell - \frac{1}{v-v_\ell} (\partial_\theta\tilde{\psi}_\ell - \partial_\theta\tilde{\psi}_g)
	= \partial_v\tilde{\eta}_\ell + \frac{1}{v-v_\ell} (\tilde{\eta}_\ell - \tilde{\eta}_g) \\
	&\partial_\theta f_2 = \partial_v \tilde{\eta}_g + \frac{1}{v-v_\ell} (\tilde{\eta}_\ell - \tilde{\eta}_g)
\end{align*}
since $\partial_\theta\tilde{p}_g=\partial_v\tilde{\eta}_g$, $\partial_\theta\tilde{p}_\ell=\partial_v\tilde{\eta}_\ell$, $\partial_\theta\tilde{\psi}_\ell=-\tilde{\eta}_\ell$, $\partial_\theta\tilde{\psi}_g=-\tilde{\eta}_g$.
 For derivative in $v$, we observe that
\begin{align*} 
	&\partial_v f_1 = 0 - Z + R(v,\theta) \\
	&\partial_v f_2 = \partial_v \tilde{p}_g + Z + R(v,\theta)
\end{align*}
with
\begin{align*}
	R(\theta,v) &= \frac{1}{(v-v_\ell)^2} (\tilde{\psi}_\ell-\tilde{\psi}_g)
	-  \frac{1}{v-v_\ell} (-\partial_v \tilde{\psi}_g) \\
	&= \frac{1}{(v-v_\ell)} \left\{\frac{\tilde{\psi}_\ell - \tilde{\psi}_g}{v-v_\ell} - \tilde{p}_g(v,\theta) \right\}.
\end{align*}
By definition of $\theta_b$ and $v_g^*$, we see that $R(\theta_b,v_g^*)=0$.
 Thus, the Jacobi matrix $J$ in $\theta$ and $v$ at $\theta_b$, $v_g^*$ equals
\[
J = \left(
\begin{array}{cc}
	\partial_\theta f_1 & \partial_v f_1 \\
	\partial_\theta f_2 & \partial_v f_2 \\
\end{array}
\right)
(\theta_b,v_g^*,0) = \left(
\begin{array}{cc}
	\partial_v\tilde{\eta}_\ell + \frac{1}{v_g^*-v_\ell}(\tilde{\eta}_\ell-\tilde{\eta}_g) & 0 \\
	\partial_v\tilde{\eta}_g + \frac{1}{v_g^*-v_\ell}(\tilde{\eta}_\ell-\tilde{\eta}_g) & \partial_v\tilde{p}_g
\end{array}
\right)
(\theta_b,v_g^*,0).
\]
We now apply the implicit function theorem to get
\begin{thm} \label{TIF}
Assume that the liquid density $\rho_\ell$ is given and $\theta_b$ is its boiling temperature, i.e., $1/\rho_\ell=v_\ell^*(\theta_b)$.
 (This means that $\rho_\ell$ is the saturated liquid density at the temperature $\theta_b$.)
 Under the regularity assumption (R) assume that
\begin{align} 
	&\partial_v \tilde{\eta}_\ell(v_\ell^*,\theta_b) > \frac{\tilde{\eta}_g^*- \tilde{\eta}_\ell^*}{v_g^*- v_\ell^*}, \label{EENR} \\
	&\partial_v \tilde{p}_g(v_g^*,\theta_b) < 0 \label{EPRR} 
\end{align}
with $\tilde{\eta}_g^*=\tilde{\eta}_g(v_g^*,\theta_b)$, $\tilde{\eta}_\ell^*=\tilde{\eta}_\ell(v_\ell^*,\theta_b)$.
 Then, for sufficiently small $Z$, there is unique $(v,\theta)$ near $(v_g^*,\theta_b)$ such that \eqref{ESE} holds.
 Moreover, the mapping $Z\mapsto\theta$ is strictly increasing by \eqref{EENR}.
\end{thm}
Our assumption on the pressure \eqref{EPRR} is consistent with previous assumptions.
 The property $\partial_v\tilde{p}_g(v_1^*,\theta_v)<0$ follows from (M1).
 By \eqref{ECC}, the inequality for the entropy \eqref{EENR} can be rewritten in the form of the saturated vapor pressure $p_*$.
 It is of the form
\begin{equation} \label{EENR1}
	\partial_\theta \tilde{p}_\ell (v_\ell^*,\theta_b)
	> \partial_\theta p_*(\theta_b)
\end{equation}
since $\partial_\theta\tilde{p}_\ell=\partial_v \tilde{\eta}_\ell$.
 This condition is consistent with (S2).
 In the language of $\psi_u$, this condition can be interpreted as $\partial_v\tilde{\eta}$ has a negative jump at $v=v_\ell^*$.
 The entropy $\tilde{\eta}$ is often assumed to be concave in $v$ as well as the monotonicity in $v$ as in (M1).
 Figure \ref{FCEN} and Figure \ref{Fent2} describe \eqref{EENR} schematically.
\begin{figure}[htb]
  \begin{minipage}[b]{0.47\linewidth}
\centering 
\includegraphics[keepaspectratio, scale=0.25]{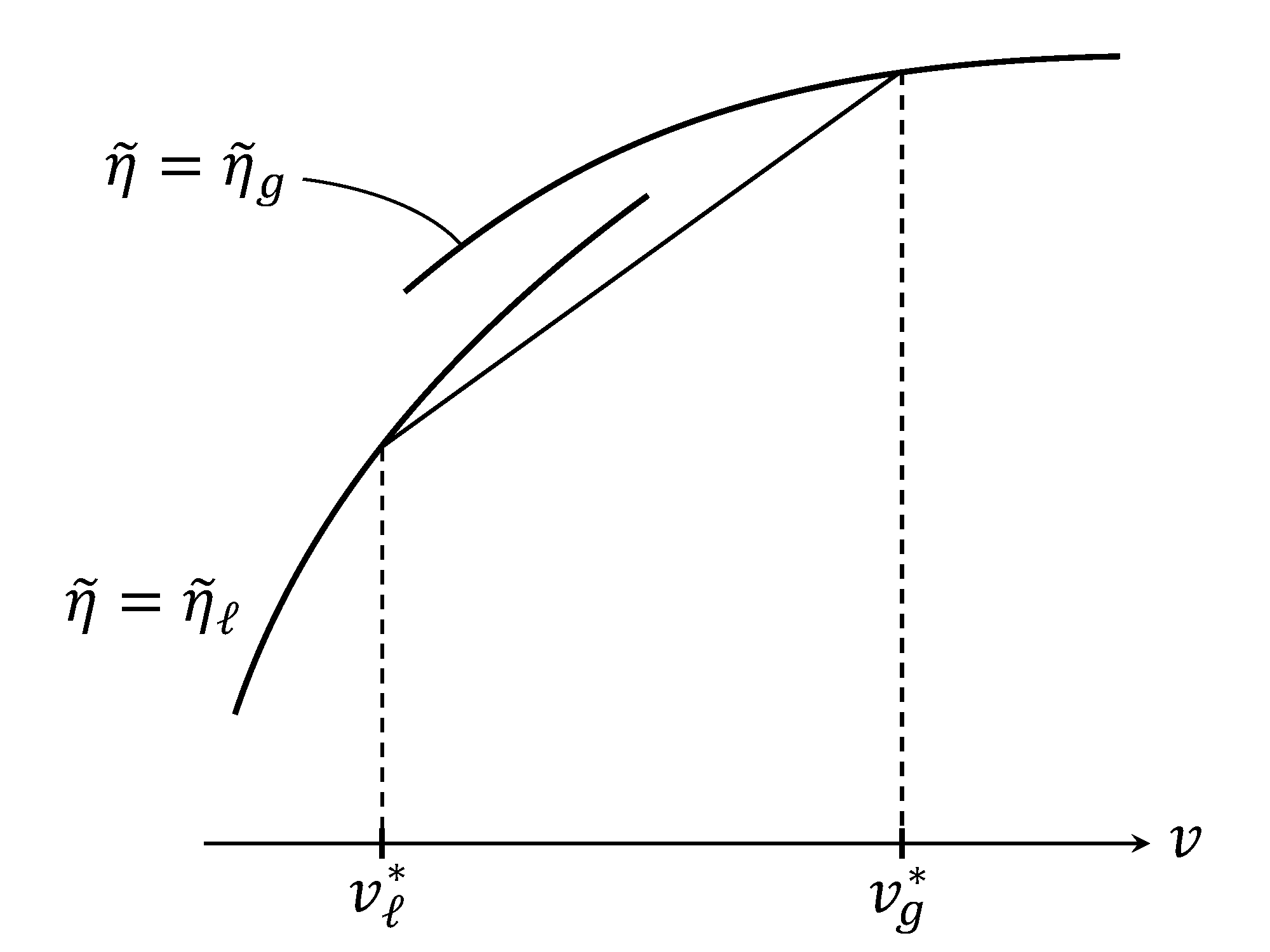} 
\caption{condition \eqref{EENR} under (M1)\label{FCEN}}
	\end{minipage}
  \begin{minipage}[b]{0.47\linewidth}
\centering 
\includegraphics[keepaspectratio, scale=0.25]{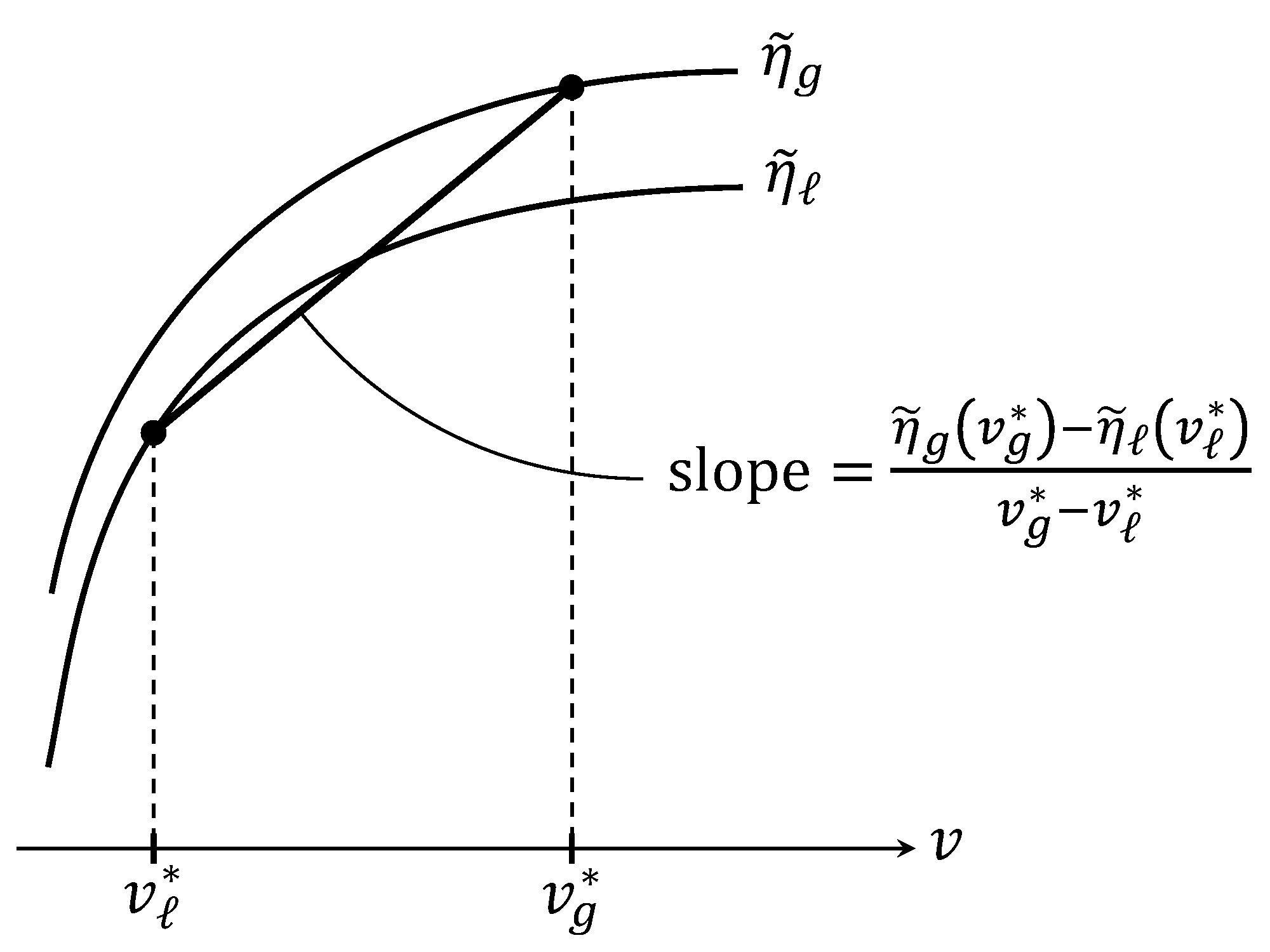}
\caption{graphs of $\tilde{\eta}_g$, $\tilde{\eta}_\ell$ at the temperature $\theta$\label{Fent2}}
  \end{minipage}
\end{figure}

It is instructive to check assumptions \eqref{EENR} for the van der Waals' Helmholtz energy of the form
\[
 \psi = k_1\theta(1-\log\theta) + k_2\theta \log \frac{\rho}{1-b\rho} -a\rho
\]
with positive constants $k_1$, $k_2$, $a$, $b$.
 Note that in the case $a=b=0$, this is the Helmholtz energy for ideal gas.
Then,
\begin{align*} 
	&\tilde{\psi}(v, \theta) = k_1\theta(1-\log\theta) - k_2\theta\log(v-b)-a/v, \\
	&\tilde{\eta}(v, \theta) = -\partial_\theta \tilde{\psi} = k_1 \log\theta + k_2 \log(v-b), \\
	&\tilde{p}(v, \theta) = -\partial_v\tilde{\psi} = k_2\theta\frac{1}{v-b} - \frac{a}{v^2}, \quad v>b.
\end{align*}
Notice that there is a non monotone part for $\tilde{p}$ with respect to $v$.
 Actually, $\tilde{\psi}$ has non convex part in $v$ for $\theta<\theta_c$ with some $\theta_c$ though it is convex in $v$ for $\theta>\theta_c$; see Figure \ref{Fvan1} and Figure \ref{Fvan2}.
\begin{figure}[htb]
  \begin{minipage}[b]{0.47\linewidth}
\centering 
\includegraphics[keepaspectratio, scale=0.25]{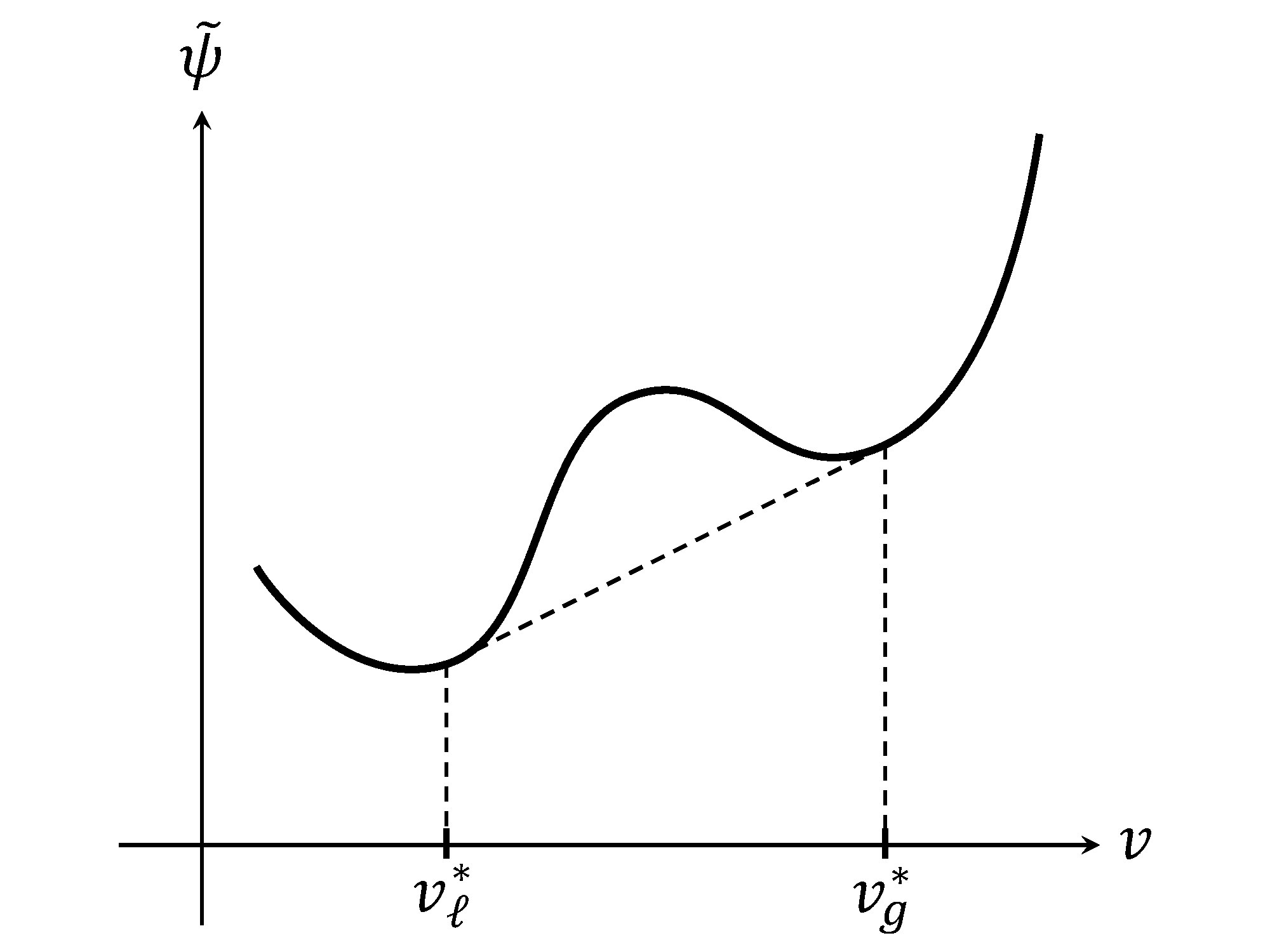} 
\caption{very low temperature\label{Fvan1}}
	\end{minipage}
  \begin{minipage}[b]{0.47\linewidth}
\centering 
\includegraphics[keepaspectratio, scale=0.25]{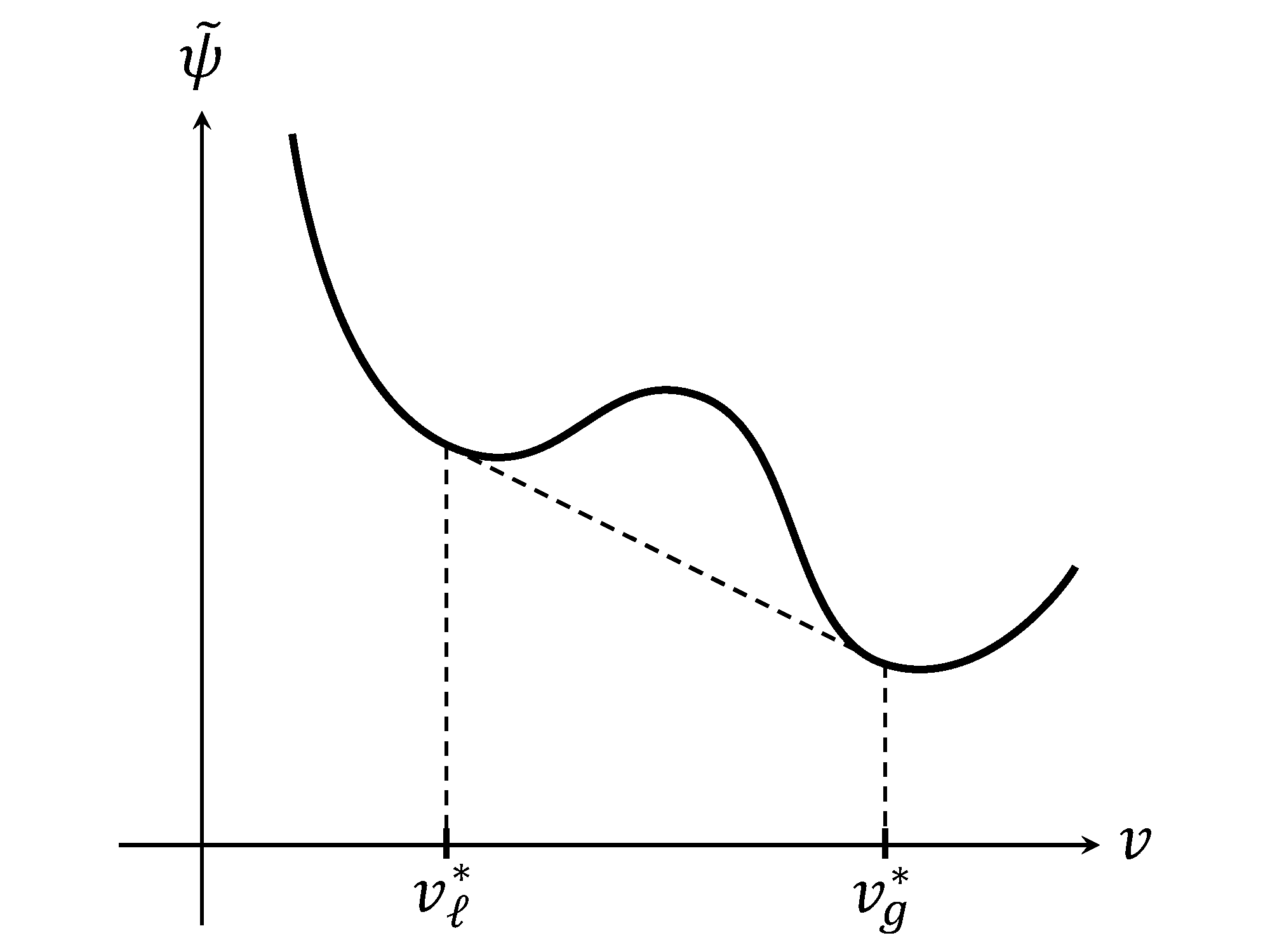}
\caption{high temperature\label{Fvan2}}
  \end{minipage}
\end{figure}
We take saturated densities $\rho_\ell^*$, $\rho_g^*$ at a given temperature as in Figure \ref{Fvan1}, \ref{Fvan2}.
The condition \eqref{EPRR} is clearly satisfied at $v_g^*$ and $v_\ell^*$.
Since $\partial_v^2 \tilde{\eta}=-k_2 / (v-b)^2<0$, the condition \eqref{EENR} is also clear; see Figure \ref{Fent}.
\begin{figure}[htb]
\centering 
\includegraphics[keepaspectratio, scale=0.25]{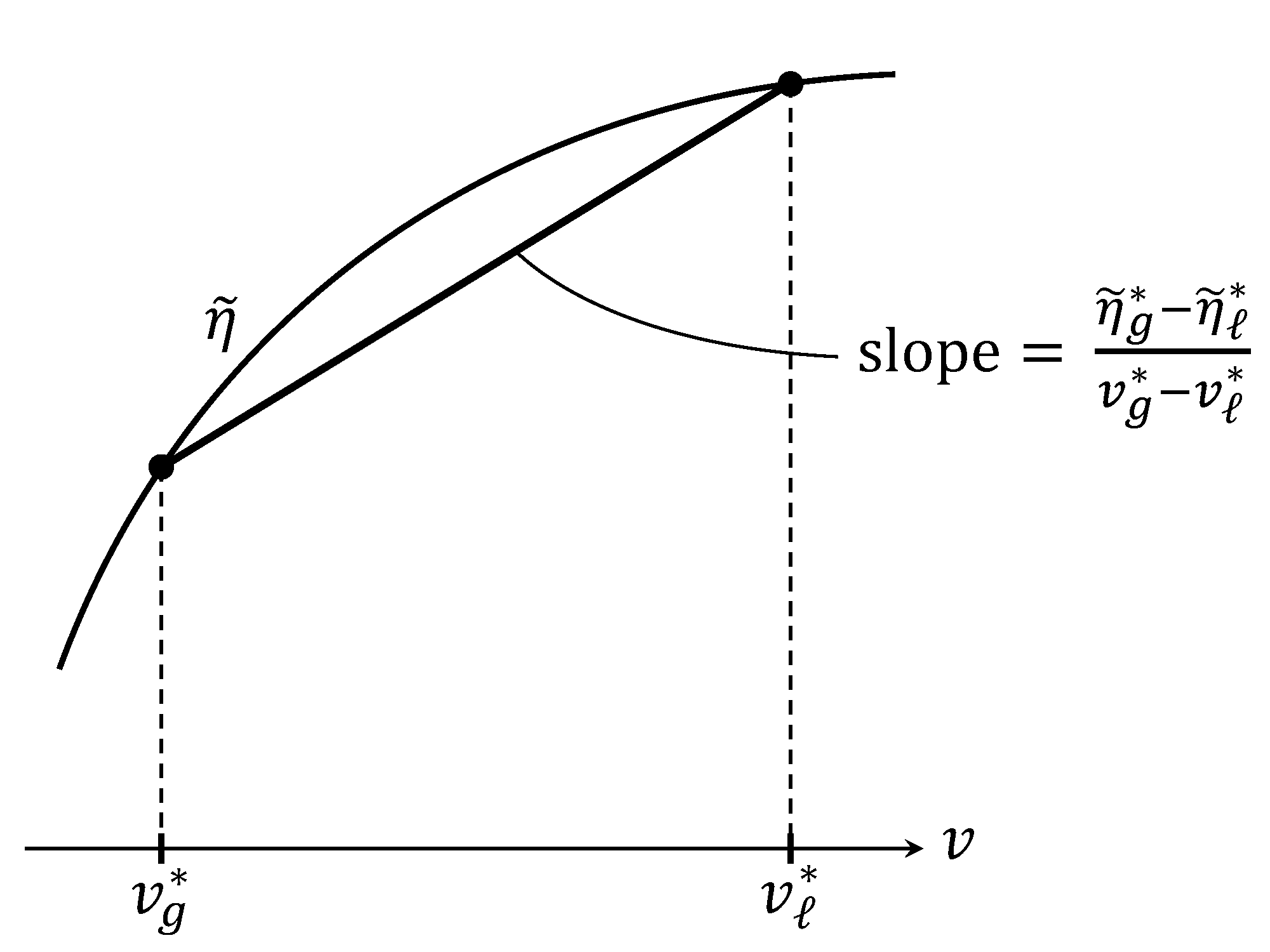} 
\caption{graph of $\widetilde{\eta}$ \label{Fent}}
\end{figure}
 By a direct and tedious calculation, we conclude that $v_g^*$, $v_\ell^*$ satisfy the monotonicity condition (S2) so $\tilde{\eta}_u$ at $v_\ell^*$ can be interpreted as $\tilde{\eta}_u(v_\ell^*,\theta_b) = - \partial_\theta \tilde{\psi}_u (v_\ell^*,\theta_b)$, where $\tilde{\psi}_u$ is the convexified $\psi$ in $v$.
 We conclude this section by noting that symmetric condition $\partial_\theta p_*(\theta_b)>\partial_\theta\tilde{p}_g(v_g^*,\theta_b)$, $\partial_v \tilde{p}_\ell(v_\ell^*,\theta_b)<0$ gives the solvability of \eqref{ESt}, \eqref{ESt2} for $v_\ell$, $\theta$ by fixing $v_g=v_g^*(\theta_b)$ provided that $Z$ is small.
 The monotonicity condition for $\theta$ also follows.

If $j_\Gamma$ is not small, there is a chance that there exist no $\theta$, $\rho_2 \neq \rho_1$ satisfying \eqref{ESE}.
To see this phenomenon, we first recall an elementary fact for a bitangent line.
\begin{lemma} \label{LBIT}
Let $f=f(\rho)$ be a $C^1$ function on some open interval $I$.
 Let $\rho_1,\rho_2\in I$ be $\rho_2<\rho_1$.
 There is a straight line through $\left(\rho_2,f(\rho_2)\right)$ and $\left(\rho_1,f(\rho_1)\right)$ which is tangent to the graph of $f$ at $\left(\rho_2,f(\rho_2)\right)$, $\left(\rho_1, f(\rho_1)\right)$ if and only if
\begin{gather*}
	\partial_\rho f(\rho_1) = \partial_\rho f(\rho_2) \\
	\rho_1 \partial_\rho f(\rho_1) - f(\rho_1) 
	= \rho_2 \partial_\rho f(\rho_2) - f(\rho_2).
\end{gather*}
This line is often called a bitangent line.
\end{lemma}

We next derive conditions equivalent to $\eqref{EI2}_1$ and \eqref{EI5}.
 If $\sigma H_\Gamma=0$ and viscous stress $S=0$, then $\eqref{EI2}_1$ and \eqref{EI5} are of the form, respectively
\begin{align}
	& \llbracket 1/\rho \rrbracket j_\Gamma^2
	+ \llbracket p \rrbracket = 0, \label{ENV1} \\
	& \llbracket \psi \rrbracket + \left\llbracket \frac{1}{2\rho^2} \right\rrbracket j_\Gamma^2
	+ \llbracket p/\rho \rrbracket = 0. \label{ENV2}
\end{align}
It is convenient to introduce the volume specific Helmholtz energy $\Psi$.
 It is of the form
\[
	\Psi = \rho\psi
\]
so that $p=\rho\partial_\rho\Psi-\Psi=\rho(\partial_\rho\Psi-\psi)$.
 In the case $j_\Gamma=0$, \eqref{ENV2} can be written
\[
	\llbracket \partial_\rho \Psi \rrbracket = 0
\]
while \eqref{ENV1} becomes
\[
	\llbracket p \rrbracket = 0.
\]
By Lemma \ref{LBIT}, this implies that there is a bitangent line to the graph of $\Psi$ from $\left(\rho_g,\Psi(\rho_g,\theta)\right)$ to $\Bigl(\rho_\ell,\Psi(\rho_\ell,\theta)\Bigr)$ for the fixed $\theta$.
 For general $j_\Gamma$, we introduce
\begin{equation} \label{EHMod}
	\psi^j = \psi - \frac{1}{2\rho^2} j_\Gamma^2, \quad
	\Psi^j = \rho\psi^j = \Psi - \frac{1}{2\rho} j_\Gamma^2
\end{equation}
and observe that
\begin{align*}
	\partial_\rho \Psi^j
	&= \psi^j + \rho \partial_\rho \psi^j
	= \psi - \frac{j_\Gamma^2}{2\rho^2} + \frac{p}{\rho} + \frac{\rho}{\rho^3} j_\Gamma^2 \\
	&= \psi + \frac{j_\Gamma^2}{2\rho^2} + \frac{p}{\rho}.
\end{align*}
Thus, \eqref{ENV2} is equivalent to
\begin{equation} \label{EEGib}
	\llbracket \partial_\rho \Psi^j \rrbracket = 0.
\end{equation}
If we set $p^j=\rho\partial_\rho\Psi^j-\Psi^j$, we observe that
\begin{align*}
	p^j 	&= \rho \left( \psi + \frac{j_\Gamma^2}{2 \rho^2} + \frac{p}{\rho} - \psi + \frac{j_\Gamma^2}{2 \rho^2} \right) \\
	&= p + \frac1\rho j_\Gamma^2.
\end{align*}
Thus, \eqref{ENV1} is equivalent to
\begin{equation} \label{EEMom}
	\llbracket p^j \rrbracket = 0.
\end{equation}
In fact, we have proved
\begin{thm} \label{TMG}
Let $\psi^j$, $\Psi^j$ be given by \eqref{EHMod}.
Then $p^j = p+(1/\rho)j_\Gamma^2$, $\partial_\rho \Psi^j = \psi+(1/2\rho^2)j_\Gamma^2+p/\rho$ if $p^j=\rho\partial_\rho\Psi^j-\Psi^j$, $p=\rho\partial_\rho\Psi-\Psi$.
In particular, \eqref{ENV1}, \eqref{ENV2} are equivalent to \eqref{EEGib} and \eqref{EEMom}.
\end{thm}

If the density of $\Omega_i$ is $\rho_i$ ($i=1,2$), then, by Lemma \ref{LBIT}, the conditions \eqref{EEGib} and \eqref{EEMom} say that there is a bitangent line of the graph of $\Psi^j$ from $\left(\rho_2,\Psi^j(\rho_2)\right)$ to $\left(\rho_1,\Psi^j(\rho_1)\right)$.
However, if $j_\Gamma$ is sufficiently large, it might happen that there is no bitangent line at least for the van der Waals' Helmholtz energy because of the term $- j_\Gamma^2/2\rho$ in the definition of $\Psi^j$.
Indeed, in the case of the van der Waals' Helmholtz energy
\[
	\Psi = k_1 \rho\theta (1-\log\theta) + k_2 \rho\theta \log\frac{\rho}{1-b\rho} - a\rho^2,
\]
we get
\[
	\partial_\rho^2 \Psi = k_2 \theta \left(\frac{b}{(1-b\rho)^2} + \frac1\rho + \frac{b}{1-b\rho} \right) - 2a.
\]
This function is positive near $\rho=0$ and $\rho$ close to $1/b$.
But it may change sign twice if $\theta$ is not very large compared with $a$.
Note that $\partial_\rho^2\Psi^j$ is negative near $\rho=0$.
If $j_\Gamma$ is not large, $\partial_\rho^2 \Psi^j$ changes its sign three times.
However, if $j_\Gamma$ is large, $\partial_\rho^2\Psi^j$ changes its sign only once no matter what the temperature $\theta>0$ is; see Figure \ref{FJvan1} and Figure \ref{FJvan2}.
\begin{figure}[htb]
  \begin{minipage}[b]{0.45\linewidth}
	\centering
	\includegraphics[keepaspectratio, scale=0.25]{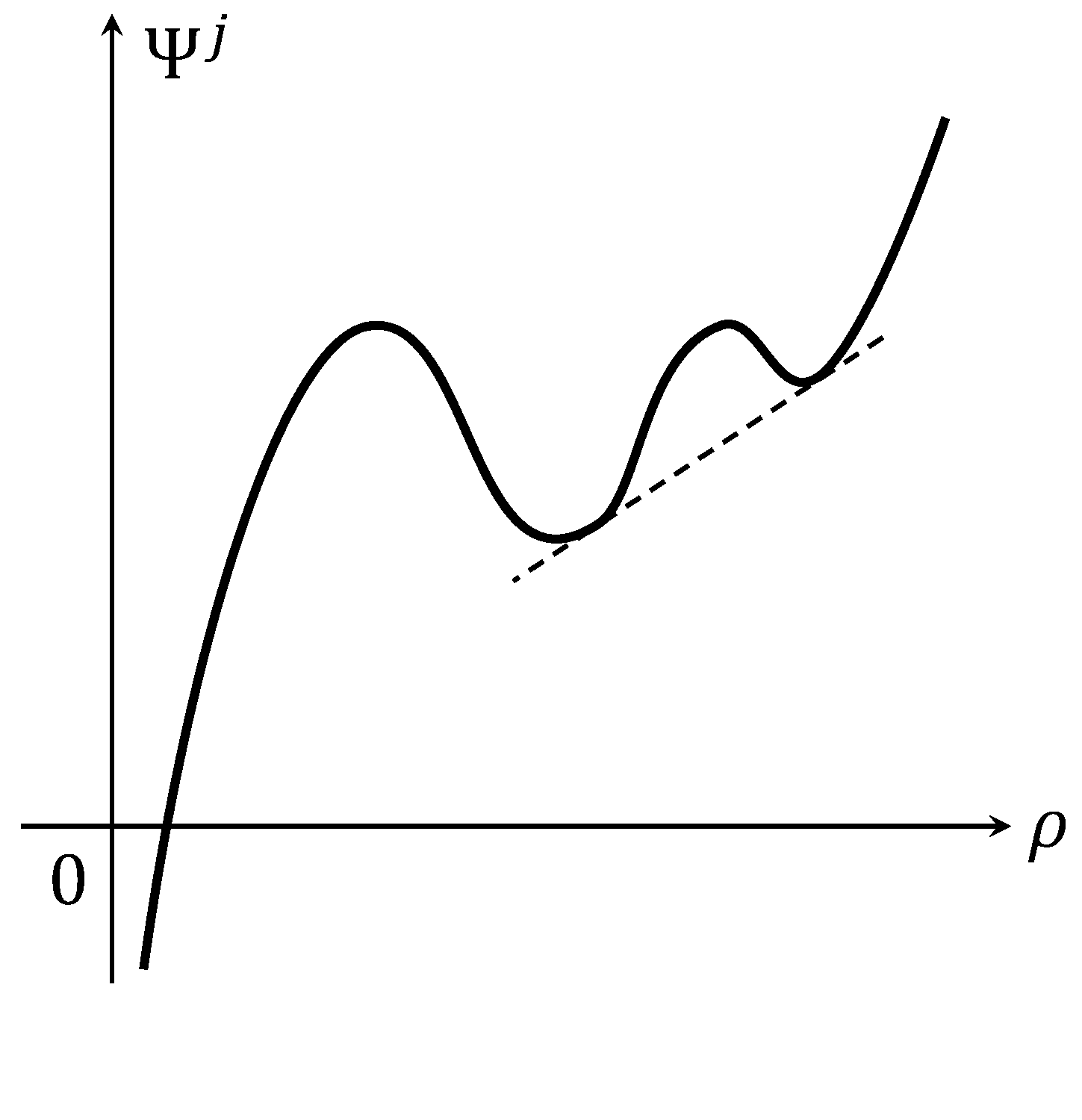}
\caption{graph of $\Psi^j$ for small $j_\Gamma$ with fixed $\theta$\label{FJvan1}}
  \end{minipage}
  \begin{minipage}[b]{0.03\linewidth}
	  \centering 
\includegraphics[keepaspectratio, scale=0.03]{blank.png}
  \end{minipage}
  \begin{minipage}[b]{0.45\linewidth}
	\centering
	\includegraphics[keepaspectratio, scale=0.25]{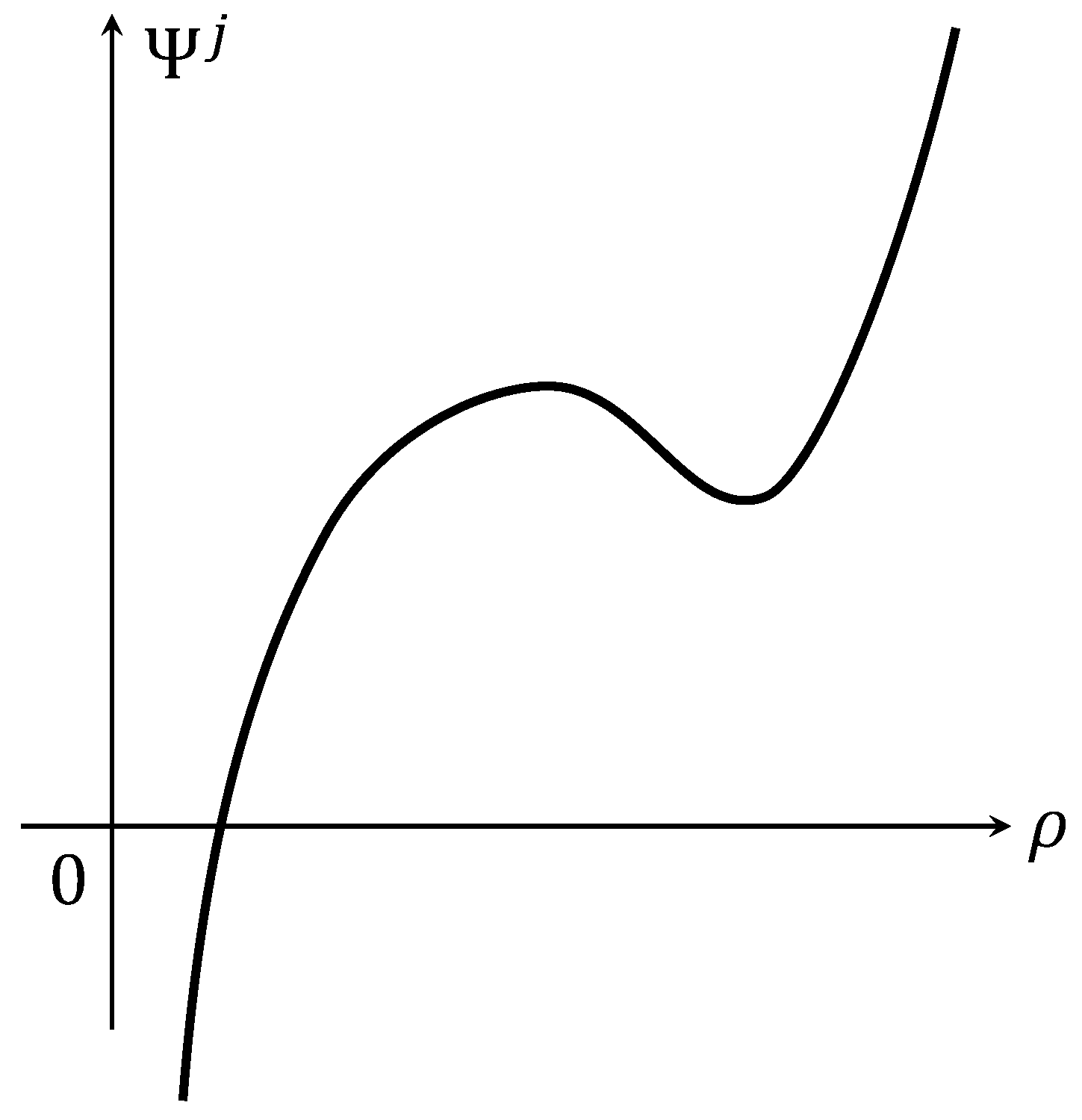}
\caption{graph of $\Psi^j$ for large $j_\Gamma$ with fixed $\theta$\label{FJvan2}}
  \end{minipage}
\end{figure}
We conclude the following theorem.
\begin{thm} \label{TVanj}
If $j_\Gamma$ is large, for any $\theta>0$, there is no pair $(\rho_1,\theta)$, $(\rho_2,\theta)$ with $\rho_1\neq\rho_2$ satisfying \eqref{ENV1} and \eqref{ENV2} for the van der Waals' energy.
\end{thm}
We conclude this section by mentioning a well-known relation between $\tilde{\psi}$ and $\Psi$.
By definition,
\[
	\partial_\rho\Psi = \rho\partial_\rho\psi + \psi
	= -v\partial_v\tilde{\psi} + \tilde{\psi}
	\quad \text{where} \quad \rho=1/v
\]
and similarly
\[
	\partial_v\tilde{\psi} = -\rho\partial_\rho \Psi + \Psi.
\]
Thus, $\llbracket \partial_\rho\Psi\rrbracket=0$, $\llbracket -\rho\partial_\rho\Psi + \Psi\rrbracket=0$ is equivalent to $\llbracket \partial_v\tilde{\psi}\rrbracket=0$, $\llbracket v\partial_v\tilde{\psi}-\tilde{\psi}\rrbracket=0$.
By Lemma \ref{LBIT}, this means that $\tilde{\psi}$ has a bitangent from $\left(v_\ell^*,\tilde{\psi}(v_\ell^*)\right)$ to $\left(v_g^*,\tilde{\psi}(v_g^*)\right)$ if and only if $\Psi$ has a bitangent from $\left(\rho_g^*,\Psi(\rho_g^*)\right)$ to $\left(\rho_\ell^*,\Psi(\rho_\ell^*)\right)$ at a fixed temperature.
We also note that the non-decreasing behavior of the pressure $p$ with respect to $\rho$ is equivalent to the convexity of $\tilde{\psi}$ in $v$ since $p=-\partial_v\tilde{\psi}$.
This also implies the convexity of $\Psi$ in $\rho$.
(In fact, it is equivalent.)
Indeed, $p=\rho\partial_\rho\Psi-\Psi$ and $\partial_\rho p=\partial_\rho^2\Psi$. 

\section{Stationary solutions to the Stefan problem} \label{S4} 

We consider a stationary solution to \eqref{EH3} with \eqref{EI2S}, \eqref{EI3S}, \eqref{EI4S}, \eqref{EI1SS} in a half line $(0,\infty)$, i.e., $L=\infty$, which was written at the end of Section \ref{S2}.
We consider the situation that the liquid region occupies near the entrance.

We give the liquid density $\rho_1=\rho_\ell$ which is a positive constant.
We also give the liquid velocity $u_1=u_\ell$ which is also a positive constant.
In Section \ref{S3}, we obtain that if the phase flux $j_\Gamma=\rho_\ell u_\ell$ is sufficiently small, there is a temperature $\theta_*$ and the gas density $\rho_g(<\rho_\ell)$ which solves $(\rho_2,\theta)$ equations of the forms
\begin{align*}
	& (\rho_1 u_1)^2 = \left(p_1(\rho_1,\theta_*) - p_2(\rho_2,\theta_*)\right) / (v_2-v_1), \quad
	v_i=1/\rho_i, \quad \rho_1=\rho_\ell, \quad \rho_2 = \rho_g, \\
	& \psi_2 (\rho_2, \theta_\ast) - \psi_1(\rho_1,\theta_*) + (v_1-v_2) \frac{p_1(\rho_1,\theta_*)+p_2(\rho_2,\theta_*)}{2} = 0
\end{align*}
under physically reasonable assumptions on the Helmholtz energy.
 We postulate that these two equations are solvable for a given $\rho_1$ and $j_\Gamma=\rho_1u_1$ to find $\rho_2=\rho_g$ and $\theta_*$.
 Since $u_2$ is determined by $u_2 = \rho_1u_1/\rho_2$, our system is reduced to
\begin{align}
	& \kappa_1 j_\Gamma \partial_{x_1} \theta_1 - \partial_{x_1}(d_1 \partial_{x_1} \theta_1) = \rho_1 r_1 &&\hspace{-9em}\text{in}\quad (0,x_*), \label{EEH1} \\
	& \kappa_2 j_\Gamma \partial_{x_1} \theta_2 - \partial_{x_1}(d_2 \partial_{x_1} \theta_2) = \rho_2 r_2 &&\hspace{-9em}\text{in}\quad (x_*,\infty), \label{EEH2} \\
	& \theta_1 = \theta_2 = \theta_* &&\hspace{-9em}\text{at}\quad x=x_*, \label{EEI3S} \\
	& \ell j_\Gamma + d_2 \partial_{x_1} \theta_2 - d_1 \partial_{x_1} \theta_1 = 0 &&\hspace{-9em}\text{at}\quad x=x_*. \label{EEI4S}
\end{align}
From physical requirement, we usually assume that $\kappa_i>0$ and $d_i>0$ for $i=1,2$.
Moreover, $\ell=-\theta_*\llbracket\eta\rrbracket<0$ if $\Omega_1$ is a liquid region and $\Omega_2$ is a vapor region.
To see the feature of the problem, we assume that these physical quantities are constants.
More precisely,
\begin{enumerate}
\item[(PC)]
The specific heat $\kappa_i(\rho_i,\theta)$ is a positive constant for $i=1,2$ and the latent heat $\ell=-\theta_*\llbracket\eta\rrbracket$ is a negative constant.
(These assumptions impose restrictions on the Helmholtz energy, especially the entropy.
 They are fulfilled, for example, for van der Waals' energy.)
The thermal diffusivity $d_i$ is a positive constant.
\end{enumerate}

We next impose the boundary condition for $\theta_i$.
At the entrance $\theta_1(0)=\theta_\mathrm{in}(<\theta_*)$ is given while at the space infinity, $\partial_x\theta_2$ is assumed to be bounded.
The second condition is reasonable if the effect of transport, i.e., the term involving $j_\Gamma$ in the second equation dominates the diffusion term.
We further assume that $\rho_1r_1=\rho_2r_2=r$ is a positive constant.
We give a necessary and sufficient condition such that the dryout point exists.
\begin{thm} \label{TUES}
Assume (PC) and that $r=\rho_1r_1=\rho_2r_2$ is a positive constant.
Assume that $\theta_*$ is given.
Then
\[
	(-\ell) \leq \frac{d_2 r}{\kappa_2 j_\Gamma^2}
\]
if and only if there exists a unique solution $(\theta_1,\theta_2,x_*)$ to \eqref{EEH1}, \eqref{EEH2}, \eqref{EEI3S}, \eqref{EEI4S} satisfying $\theta_1\leq\theta_*$ and $\theta_1(0)=\theta_\mathrm{in}<\theta_*$ under the linear growth assumption on $\theta_2$ as $x\to\infty$.
Moreover, the function $x_*(\theta_\mathrm{in},j_\Gamma)$ is strictly decreasing in $\theta_\mathrm{in}$ and strictly increasing in $j_\Gamma$.
Furthermore $x_*\to\infty$ as $j_\Gamma\to\infty$ or $\theta_\mathrm{in}\to-\infty$ if one fixes the remaining variable. 
\end{thm}
%
For the proof, we prepare a simple lemma on a free boundary problem for a second-order ordinary differential equation (ODE) of the form
\begin{align} 
\begin{aligned} \label{EFPD}
\left\{
\begin{array}{l}
	by' - ay'' - c = 0, \quad 0<x<\hat{x}, \\
	y(\hat{x}) = 0, \\
	y(0) = -y_0, \\
	y'(\hat{x}) = z_0.
\end{array}
\right.
\end{aligned}
\end{align}
\begin{lemma} \label{LUEFP}
Assume that $a$, $b$ and $c$ are positive constants.
For given positive constants $y_0$ and $z_0$, there exists a unique $(y,\hat{x})$ with $\hat{x}>0$ which solves \eqref{EFPD}.
Moreover, $\hat{x}=\hat{x}(y_0,z_0)$ is strictly increasing in $y_0$ and strictly decreasing in $z_0$.
Furthermore, there is $\hat{x}_c(y_0)>0$ such that
\begin{gather*}
	\lim_{y_0\to\infty} \hat{x}(y_0,z_0) = \infty, \quad
	\lim_{y_0\downarrow0} \hat{x}(y_0,z_0) = 0, \quad\text{for any}\quad z_0>0, \\
		\lim_{z_0\to\infty} \hat{x}(y_0,z_0) = 0, \quad
	\lim_{z_0\downarrow0} \hat{x}(y_0,z_0) = \hat{x}_c(y_0) \quad\text{for any}\quad y_0>0. 
\end{gather*}
The function $\hat{x}_c(y_0)$ is strictly increasing and
\[
	\lim_{y_0\to0} \hat{x}_c (y_0) = 0, \quad
	\lim_{y_0\to\infty} \hat{x}_c (y_0) = \infty.
\]
All dependence with respect to $y_0$ and $z_0$ are analytic.
\end{lemma}
\begin{proof}
We may assume that $c=1$ by dividing both sides of the first equation of \eqref{EFPD} by $c$.
By scaling $y_b(x)=y(bx)$, we may assume $b=1$.
The resulting ODE is
\[
	y' - ay'' - 1 =0
\]
and its general solution is of the form
\[
	y(x) = x + c_2 e^{x/a} +c_1
\]
with constant $c_1$ and $c_2$.
We shall determine $c_1$, $c_2$, $\hat{x}$ so that $y(\hat{x})=0$, $y(0)=-y_0$, $y'(\hat{x})=z_0$, namely
\begin{align*}
	0 &= \hat{x} + c_2 e^{\hat{x}/a} + c_1, \\
	-y_0 &= c_1 + c_2, \\
	z_0 &= 1 + \frac{c_2}{a} e^{\hat{x}/a}.
\end{align*}
From the first two equations,
\[
	c_2 = \frac{y_0 - \hat{x}}{e^{\hat{x}/a} -1}.
\]
Thus,
\begin{equation} \label{EZF}
	z_0 = 1 + \frac{y_0-\hat{x}}{a(e^{\hat{x}/a}-1)} e^{\hat{x}/a}.
\end{equation}
We shall fix $y_0>0$.
As a function of $\hat{x}$, $\frac{dz_0}{d\hat{x}}<0$, in particular, $z_0$ is strictly decreasing in $\hat{x}$.
Indeed, we set $\tau=e^{\hat{x}/a}-1$ so that $\hat{x}=a\log(\tau+1)$.
By \eqref{EZF}, $z_0(\hat{x})$ is of the form
\[
	z_0(\hat{x}) = 1+\frac1a f(\tau), \quad
	f(\tau) = \left(y_0 - a \log(\tau+1)\right)(\tau+1)/\tau.
\]
Since
\begin{align*}
	f'(\tau) &= -\frac{a}{\tau+1} \left( \frac{\tau+1}{\tau} \right) - \frac1\tau^2 \left( y_0 - a\log(\tau+1) \right) \\
	&= -\frac1\tau \left( a + \frac{y_0 - a\log(\tau+1)}{\tau} \right)
\end{align*}
and $\tau-\log(\tau+1)\geq0$, we observe that
\[
	f'(\tau) \leq -\frac1\tau \frac{y_0}{\tau} < 0.
\]
Thus $\frac{dz_0}{d\hat{x}}<0$ everywhere.

From \eqref{EZF}, it follows that
\begin{equation} \label{EZP}
	\lim_{\hat{x}\downarrow0} z_0 (\hat{x}) = \infty, \quad
	\lim_{\hat{x}\to\infty} z_0 (\hat{x}) = -\infty.
\end{equation}
By strictly monotonicity of $z_0$ in $\hat{x}$, we now conclude that for a given $y_0$ and $z_0$, there is a unique $\hat{x}>0$ and $y$ satisfying \eqref{EFPD}.
 The function $\hat{x}_c(y_0)$ is defined as $z_0=0$ in \eqref{EZF}, i.e.,
\[
	1 + \frac1a \frac{y_0 - \hat{x}_c}{(e^{\hat{x}_c/a}-1)} e^{\hat{x}_c/a} = 0.
\]
The properties of $\hat{x}(y_0,z_0)$ in Lemma \ref{LUEFP} for fixed $y_0>0$ easily follows by monotonicity of $z_0$ and \eqref{EZP}.
Regularity issues are clear by explicit formulas.

It remains to study the dependence of $y_0$ in $ \hat{x}(y_0,z_0)$ for a fixed $z_0$.
By \eqref{EZF}, we see that
\begin{equation} \label{EZFF}
	\frac{(z_0-1)\tau}{\tau+1} + \log(\tau+1) = \frac{y_0}{a},
\end{equation}
where $\tau=e^{\hat{x}/a}-1$.
Differentiating the left-hand side in $\tau$, we get
\[
	\frac{(z_0-1)}{(\tau+1)^2} + \frac{1}{\tau+1} \geq \frac{1}{\tau+1} \left(1-\frac{1}{\tau+1} \right)
	\quad\text{for}\quad z_0 \geq 0,\ \tau>0.
\]
Thus, strictly monotonicity property of $\hat{x}$ and $\hat{x}_c$ with respect to $y_0$ is obtained since the left hand side of \eqref{EZFF} is strictly increasing and $0$ at $\tau=0$ and $\infty$ as $\tau\to\infty$.
(Note that $\hat{x}_c(y_0)=\hat{x}(y_0,0)$.)

The graphs of solutions to (\ref{EFPD}) are illustrated in Figure \ref{FGSoL}. 
\begin{figure}[htbp]
\centering 
\includegraphics[keepaspectratio, scale=0.3]{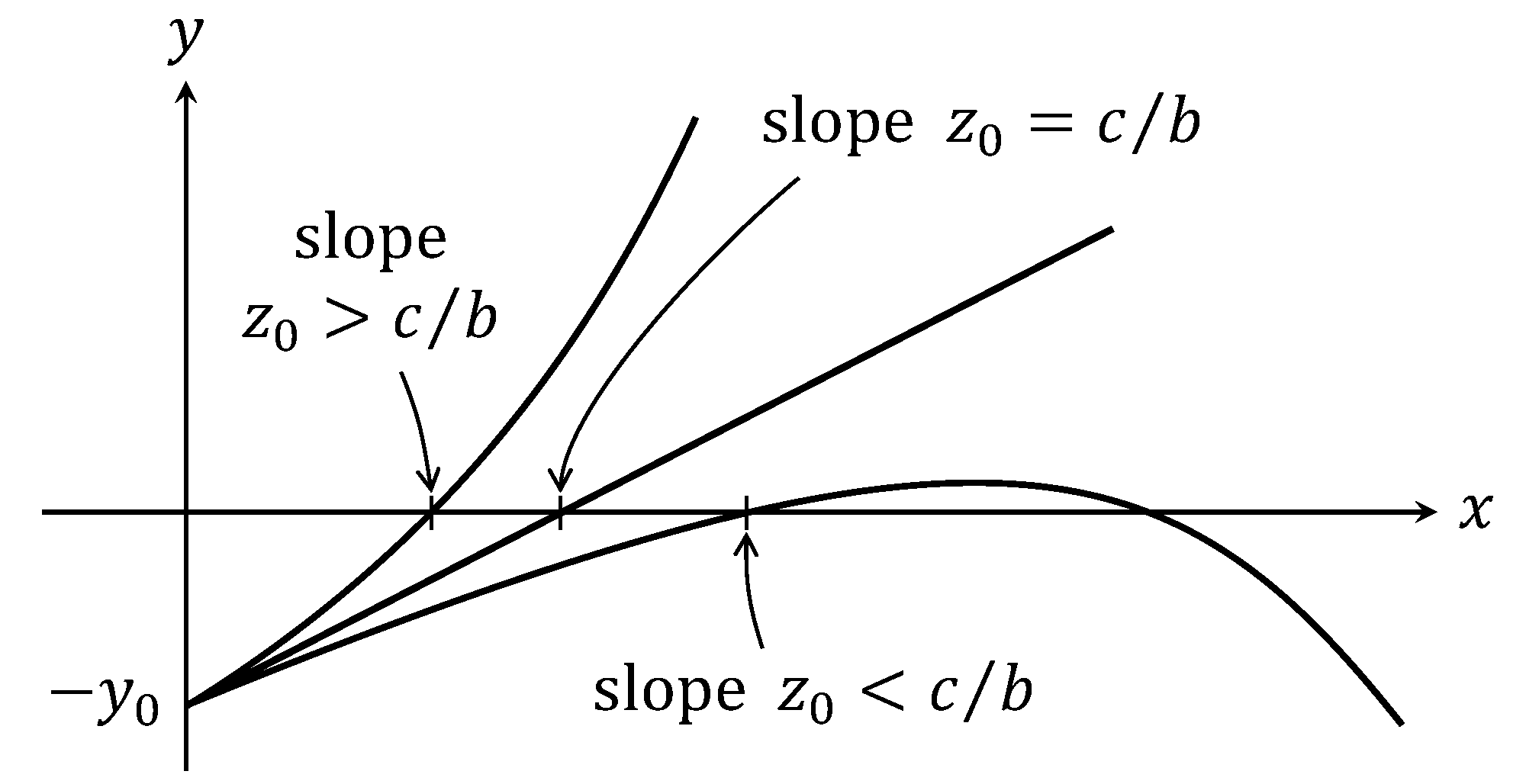} 
\caption{graphs of solutions\label{FGSoL}}
\end{figure}
\end{proof}
%
\begin{proof}[Proof of Theorem \ref{TUES}]
If $\theta_2$ solves \eqref{EEH2}, then it is of the form
\[
	\theta_2(x_1) = \frac{r}{\kappa_2 j_\Gamma} x_1 + c_1 + c_2 e^{\alpha x_1}, \quad
	\alpha = \frac{\kappa_2 j_\Gamma}{d_2}
\]
with some constant $c_1$ and $c_2$.
Under the linear growth assumption, $c_2=0$.
Thus
\[
	\theta_2(x_1) = \frac{r}{\kappa_2 j_\Gamma} (x_1 - x_*) + \theta_*
\]
if $\theta_2(x_*)=\theta_*$.
The Stefan condition \eqref{EEI4S} is now of the form
\begin{align*}
	(-\ell) j_\Gamma &= d_2 \partial_{x_1} \theta_2 - d_1\partial_{x_1}\theta_1 \\
	&= \frac{d_2 r}{\kappa_2 j_\Gamma} - d_1\partial_{x_1}\theta_1.
\end{align*}
Note that $\partial_{x_1}\theta_1(x_*)\geq0$ since $\theta_1\leq\theta_*$.
Thus, 
if a solution $(\theta_1,\theta_2,x_*)$ exists, we must have
\[
-\ell \leq \frac{d_2 r}{\kappa_2 j_\Gamma^2}.
\]

It remains to prove that this condition is sufficient to have a solution.
We may assume $\theta_*=0$ by considering $\theta-\theta_*$.
Our problem is reduced to Lemma \ref{LUEFP} by setting
\begin{align*} 
	&b= \kappa_1 j_\Gamma, \quad a = d_1, \quad c = r, \quad y_0 = \theta_\mathrm{in}, \\
	&z_0 = \frac{1}{d_1} \left( \frac{d_2 r}{\kappa_2 j_\Gamma} + \ell j_\Gamma \right).
\end{align*}
If we set $x_*=\hat{x}$, then all required properties follows from Lemma \ref{LUEFP}.
\end{proof}
\vspace{0.5cm}
\noindent
{\bf Acknowledgments.}\
This work was done as a part of research activities of Social Cooperation Program ``Mathematical Sciences for Refrigerant Thermal Fluids'' sponsored by Daikin Industries, Ltd.\ at the University of Tokyo. The authors are grateful to members of the Technology and Innovation Center of Daikin Industries, Ltd.\ for showing several interesting phenomena related to dryout points with fruitful discussion which triggered this work.
 The work of the first author was partly supported by the Japan Society for the Promotion of Science (JSPS) through the grants Kakenhi: No.~20K20342, No.~19H00639, and by Arithmer Inc., Daikin Industries, Ltd.\ and Ebara Corporation through collaborative grants.

\end{document}